\documentclass[
final
, nomarks
]{dmtcs-episciences}

\usepackage[utf8]{inputenc}
\usepackage{subfigure}

\usepackage{amssymb,amsmath,amsthm,latexsym,enumerate}
\usepackage{hyperref,url}
\usepackage{multicol,enumitem}
\usepackage{verbatim,color}
\usepackage[raggedright]{titlesec} 

\newtheorem{thm}{Theorem}
\numberwithin{thm}{section}
\newtheorem{cor}[thm]{Corollary}
\newtheorem{lem}[thm]{Lemma}
\newtheorem{prop}[thm]{Proposition}

\newtheorem{conj}[thm]{Conjecture}

\theoremstyle{definition}
\newtheorem{defin}[thm]{Definition}

\newtheorem{exa}[thm]{Example}

\newcommand\bI{\mathbf{I}}
\newcommand\Tstrut{\rule{0pt}{2.6ex}}         
\newcommand\Bstrut{\rule[-0.9ex]{0pt}{0pt}}   

\DeclareMathOperator{\Em}{Em}

\author{Juan S. Auli
  \and Sergi Elizalde\thanks{Partially supported by Simons Foundation grant \#280575.}}
\title{Consecutive Patterns in Inversion Sequences}
\affiliation{
Department of Mathematics, Dartmouth College
}
\keywords{Inversion sequence, pattern avoidance, consecutive pattern, Wilf equivalence}
\received{2019-4-5}
\revised{2019-9-17}
\accepted{2019-10-15}

\begin{document}
\publicationdetails{21}{2019}{2}{6}{5350}
\maketitle
\begin{abstract}
An inversion sequence of length $n$ is an integer sequence
$e=e_{1}e_{2}\dots e_{n}$
such that $0\leq e_{i}<i$ for each $i$.
Corteel--Martinez--Savage--Weselcouch and Mansour--Shattuck began the study of
patterns in inversion sequences, focusing on the
enumeration of those that avoid classical patterns of length 3.
We initiate an analogous systematic study of {\em consecutive} patterns in inversion
sequences, namely patterns whose entries are required to occur in adjacent positions. We
enumerate inversion sequences that avoid consecutive patterns of length 3, and generalize some results to patterns of arbitrary length. Additionally, we study the notion of Wilf equivalence of consecutive patterns in inversion
sequences, as well as generalizations of this notion analogous to those studied for permutation patterns.
We classify patterns of length up to 4 according to the corresponding Wilf equivalence relations.
\end{abstract}

\section{Introduction}\label{sec:intro}

Let $\bI_n$ denote the set of inversion sequences of length $n$, defined as
integer sequences $e=e_{1}e_{2}\dots e_{n}$ such that $0\leq e_{i}<i$ for each $i$. Let $S_n$ denote
the symmetric group of all permutations of $[n]=\{1,2,\dots,n\}$.
There is a natural bijection between $S_n$ and $\bI_n$, where each
permutation $\pi\in S_n$ is encoded by its inversion sequence
\begin{equation}\label{eq:theta_bijection}
\Theta(\pi)=e=e_{1}e_{2}\dots e_{n}, \quad\textnormal{ with }\quad
e_{i}=\left|\{j:j<i\textnormal{ and }\pi_{j}>\pi_{i}\}\right|.
\end{equation}
Other closely-related bijections between $S_{n}$ and $\bI_{n}$, obtained by composing $\Theta$ with trivial symmetries, appear in the literature as well.
Corteel, Martinez, Savage and Weselcouch~\cite{MartinezSavageI}, and
Mansour and Shattuck~\cite{MansourShattuck}
initiated the
study of classical patterns in inversion sequences, motivated by the potential of this correspondence to
inform the study of patterns in permutations. Their systematic
enumeration of classical patterns of length 3 turns out to be interesting in its own right, as
it connects these patterns in inversion sequences to well-known sequences such as Bell numbers,
Euler up/down numbers, Fibonacci numbers and Schr{\"o}der numbers.
Since the appearance of the work of Corteel {\it et al.} and Mansour and Shattuck, the interest in patterns in inversion sequences has been on the rise, and research in this area continues to proliferate~\cite{Beaton, Bouvel, KimLin, KimLinII, Lin, MartinezSavageII, Yan}.

Let us start with some basic definitions. We define a {\em pattern} to be a sequence $p=p_{1}p_{2}\dots p_{r}$ with $p_{i}\in\{0,1,\ldots,r-1\}$ for all $i$, where a value $j>0$ can appear in $p$ only if $j-1$ appears as well.
Define the {\em reduction} of a word $w=w_1w_2\dots w_k$ over the integers to be the word obtained by
replacing all the occurrences of the $i$th smallest entry of $w$ with $i-1$ for all $i$. We say that an inversion sequence $e$ {\em contains} the
classical
pattern $p=p_{1}p_{2}\dots p_{r}$ if there exists
a subsequence
$e_{i_{1}}e_{i_{2}}\dots e_{i_{r}}$ of $e$ whose reduction is $p$.
Otherwise, we say that $e$ {\em avoids} $p$.

\begin{exa}
The inversion sequence $e=0021213$ avoids the classical pattern $210$, but it contains
the classical pattern $110$. Indeed, $e_{3}e_{5}e_{6}=221$ has reduction $110$.
\end{exa}

Unlike patterns in permutations, patterns in inversion sequences may have repeated entries.

Corteel {\it et al.}~\cite{MartinezSavageI} and Mansour and
Shattuck~\cite{MansourShattuck} provide formulas for the number of inversion sequences
of length $n$ that avoid a given pattern $p$ of length 3.
In this paper, we initiate an analogous systematic study of consecutive patterns in inversion sequences.

We define consecutive patterns in inversion sequences in analogy to consecutive patterns in permutations, which were
introduced by Elizalde and Noy~\cite{ElizaldeNoy}.
We underline the entries of a consecutive pattern to distinguish it from the classical case.

\begin{defin} An inversion sequence $e$ {\em contains} the consecutive pattern
$p=\underline{p_{1}p_{2}\dots p_{r}}$ if there is a consecutive subsequence
$e_{i}e_{i+1}\dots e_{i+r-1}$ of $e$ whose reduction is $p$; this subsequence is called an {\em occurrence} of $p$ in position $i$.
Otherwise, we say that $e$ {\em avoids} $p$.  Define
\[
\Em(p,e)=\left\{i:e_{i}e_{i+1}\dots e_{i+r-1} \textnormal{ is an occurrence of } p\right\}.
\]
Denote by $\bI_n(p)$ the set of inversion sequences of length $n$ that avoid $p$.
\end{defin}

\begin{exa}\label{exa:contain_avoid_consec}
The inversion sequence $e=0021100300\in\bI_{10}$ avoids  $\underline{120}$, but it contains $\underline{010}$, since $e_7e_8e_9=030$ is an occurrence of $\underline{010}$. It also contains three occurrences of $p=\underline{100}$, in positions $\Em(p,e)=\left\{3,5,8\right\}$.
\end{exa}

It is often useful to represent an inversion sequence $e$ as an underdiagonal lattice path consisting of unit horizontal and vertical steps. Each entry $e_{i}$ of $e$ is represented by a horizontal step: a segment between the points $(i-1,e_{i})$ and $(i,e_{i})$. The vertical steps are then inserted to make the path connected (see Figure~\ref{fig:path_represent}).

\begin{figure}[htb]
	\noindent \begin{centering}
	\includegraphics[width=0.323\textwidth]{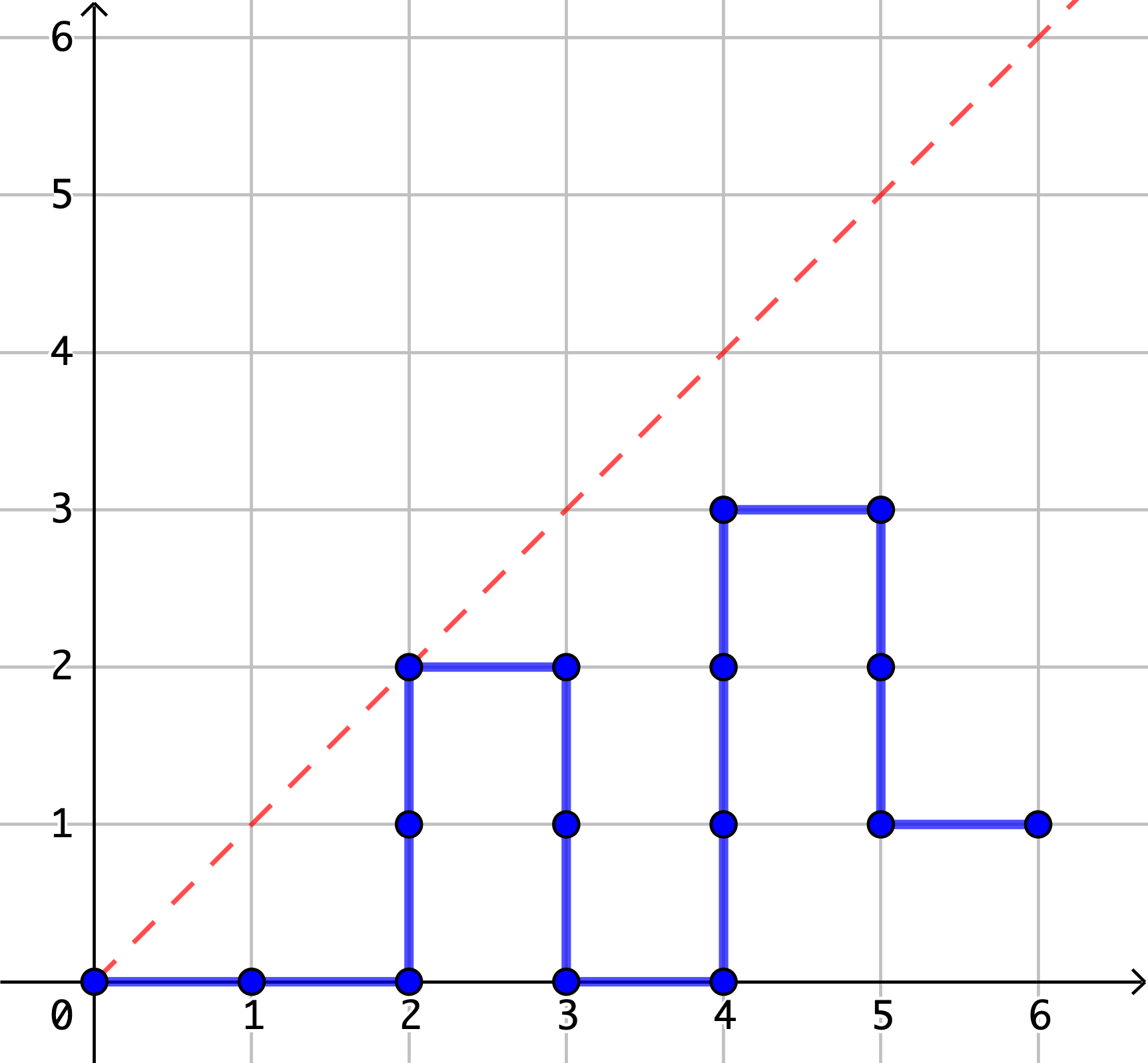}
		\par\end{centering}

	\protect\caption{Representation of $e=002031\in\bI_{6}$ as a lattice path.\label{fig:path_represent}}
\end{figure}

We are also interested in the following equivalence relations, defined in analogy to those for consecutive patterns in permutations (see~\cite{DwyerElizalde,Elizalde}). Henceforth, unless otherwise stated,
patterns will refer to patterns in inversion sequences, and
equivalence of patterns will refer to the equivalences in Definition~\ref{def:Wilf}.

\begin{defin}\label{def:Wilf} Let $p$ and $p'$ be consecutive patterns. We say that $p$ and $p'$ are
\begin{itemize}
\item
{\em Wilf equivalent\/}, denoted by $p\sim p'$, if
$\left|\bI_{n}\left(p\right)\right|=\left|\bI_{n}\left(p'\right)\right|$,
for all $n$;
\item
{\em strongly Wilf equivalent\/}, denoted by $p\stackrel{s}{\sim}p'$,
if for each $n$ and $m$, the number of
inversion sequences in
$\bI_{n}$ containing $m$ occurrences of $p$ is the same as for $p'$;
\item {\em super-strongly Wilf equivalent}, denoted  by
$p\stackrel{ss}{\sim}p'$, if the above condition holds
for any set of prescribed positions for the $m$ occurrences; that is,
\[
\left|\left\{e\in\bI_{n}:\Em(p,e)=T\right\}\right|=\left|\left\{e\in\bI_{n}:\Em(p',e)=T\right\}\right|.
\]
for all $n$ and all $T\subseteq [n]$.
\end{itemize}
\end{defin}

As suggested by their names, we have that $p\stackrel{ss}{\sim}p'$ implies $p\stackrel{s}{\sim}p'$, which in turn implies $p\sim p'$.
An equivalence of any one of the these three types will be called a {\em generalized Wilf equivalence}.

\section{Summary of Results}

In this paper we give recurrences to compute the numbers $|\bI_{n}(p)|$ of inversion
sequences avoiding the consecutive pattern $p$, for $p$ of length 3. Most of our recurrences
use the refinement $\bI_{n,k}=\{e\in\bI_{n}:e_{n}=k\}$ of
$\bI_{n}$, which allows us to define
\begin{equation}\label{eq:Ink}
\bI_{n,k}\left(p\right)=\bI_{n,k}\bigcap \bI_{n}\left(p\right),
\end{equation}
for a given pattern $p$. Note that
$\bI_{n,k}\left(p\right)=\emptyset$ for $k\geq n$ and
$\bI_{n}\left(p\right)=\bigcup_{k=0}^{n-1}\bI_{n,k}\left(p\right)$, for any pattern~$p$.

\begin{exa} For the consecutive pattern $p=\underline{001}$, we have
$\left|\bI_{3}\left(\underline{001}\right)\right|=4$, as
$\bI_{3,0}\left(\underline{001}\right)= \{000, 010\}$,
$\bI_{3,1}\left(\underline{001}\right)=\{011\}$ and
$\bI_{3,2}\left(\underline{001}\right)=\{012\}$.
\end{exa}

\begin{table}[htb]
\footnotesize
\centering
 \begin{tabular}{ c@{\hskip 0.2in}c@{\hskip 0.2in}l@{\hskip 0.2in}l }
 \hline
 Pattern $p$ & in the OEIS? & Recurrence for $\left|\bI_{n,k}(p)\right|$ & $\left|\bI_{n}(p)\right|$ coincides with
 \Tstrut\Bstrut\\
 \hline
 $\underline{012}$ & A049774 & $\left|\bI_{n,k}\left(p\right)\right|=\left|\bI_{n-1}\left(p\right)\right|-\sum_{l= 1}^{k-1}\sum_{j= 0}^{l-1}\sum_{i\geq j}\left|\bI_{n-3,i}\left(p\right)\right|$ &
 $\left|S_{n}\left(\underline{321}\right)\right|$ \Tstrut\\
 $\underline{021}$ & A071075 & $
\left|\bI_{n,k}\left(p\right)\right|=\left|\bI_{n-1}\left(p\right)\right|-
(n-2-k)\sum_{j= 0}^{k-1}\left|\bI_{n-2,j}\left(p\right)\right|$
& $\left|S_{n}\left(\underline{132}
 4\right)\right|$  \Tstrut\\
 $\underline{102}$ & New &
 $\left|\bI_{n,k}\left(p\right)\right|=\left|\bI_{n-1}\left(p\right)\right|-
 \sum_{j\geq1}j\left|\bI_{n-2,j}\left(p\right)\right|$ \Tstrut\\
 $\underline{120}$ & A200404 & $\left|\bI_{n,k}\left(p\right)\right|=\left|\bI_{n-1}\left(p\right)\right|-\sum_{j>k}(n-2-j)\left|\bI_{n-2,j}\left(p\right)\right|$ & $\left|S_{n}\left(\underline{143}
 2\right)\right|$ \Tstrut\\
 $\underline{201}$ & New &
 $\left|\bI_{n,k}\left(p\right)\right|=\left|\bI_{n-1}\left(p\right)\right|-k
 \sum_{j>k}\left|\bI_{n-2,j}\left(p\right)\right|$ \Tstrut\\
 $\underline{210}$ & New &
 $\left|\bI_{n,k}\left(p\right)\right|=\left|\bI_{n-1}\left(p\right)\right|-
 \sum_{l=k+1}^{n-4}\sum_{j=l+1}^{n-3}\sum_{i\leq j}\left|\bI_{n-3,i}(p)\right|$ \Tstrut\Bstrut\\
 \hline
 \end{tabular}
\caption{Consecutive patterns of length 3 with no repeated letters. Sequences A049774, A071075 and
A200404 are known to count $\left|S_{n}\left(\underline{321}\right)\right|$, $\left|S_{n}\left(\underline{132}
4\right)\right|$ and $\left|S_{n}\left(\underline{143}
2\right)\right|$, respectively, and exponential generating functions for A049774 and A071075 are known. All other results in the table are new.
}\label{tab1}
\end{table}

A summary of our results for consecutive patterns of length 3 is given in Table~\ref{tab1}
for patterns with no repeated letters, and in Table~\ref{tab2}
for patterns with repeated letters. In particular, we indicate which sequences
match existing sequences in the
On-Line Encyclopedia of Integer Sequences (OEIS)~\cite{OEIS}.
Here, $\underline{132}4$ and $\underline{143}2$ denote {\em vincular} (sometimes called {\em generalized}) permutation patterns, where the entries in underlined positions are required to be adjacent in an occurrence, as defined by Babson and Steingr\'{\i}msson~\cite{Babson}.
We use $S_n(\sigma)$ to denote the set of permutations in $S_n$ that avoid a pattern~$\sigma$.

\begin{table}[htb]
\footnotesize
\centering
 \begin{tabular}{ c@{\hskip 0.2in}c@{\hskip 0.2in}l }
 \hline
 Pattern $p$ & in the OEIS? & Recurrence for $\left|\bI_{n,k}(p)\right|$
 \Tstrut\Bstrut\\
 \hline
 $\underline{000}$ & A052169 \Tstrut & $\left|\bI_{n}(p)\right|=\frac{\left(n+1\right)!-d_{n+1}}{n}$,
 where $d_{n}$ is the number of derangements of [n] \\
 $\underline{001}$ & New &
 $\left|\bI_{n,k}\left(p\right)\right|=\left|\bI_{n-1}\left(p\right)\right|-
 \sum_{j<k}\left|\bI_{n-2,j}\left(p\right)\right|$ \Tstrut\\
 $\underline{010}$ & New &
 $\left|\bI_{n,k}\left(p\right)\right|=\left|\bI_{n-1}\left(p\right)\right|-
 \left(n-2-k\right)\left|\bI_{n-2,k}\left(p\right)\right|$ \Tstrut\\
 $\underline{011}$ & New\Tstrut &
 $\left|\bI_{n,k}\left(p\right)\right|=\left|\bI_{n-1}\left(p\right)\right|-
 \sum_{j<k}\left|\bI_{n-2,j}\left(p\right)\right|$ if $k\neq n-1$, and
 $\left|\bI_{n,n-1}\left(p\right)\right|=\left|\bI_{n-1}(p)\right|$ \\
 $\underline{100} \stackrel{ss}{\sim} \underline{110}$ & New &
 $\left|\bI_{n,k}\left(p\right)\right|=\left|\bI_{n-1}\left(p\right)\right|-
 \sum_{j>k}\left|\bI_{n-2,j}\left(p\right)\right|$ \Tstrut\\
 $\underline{101}$ & New &
 $\left|\bI_{n,k}\left(p\right)\right|=\left|\bI_{n-1}\left(p\right)\right|-
 k\left|\bI_{n-2,k}\left(p\right)\right|$ \Tstrut\Bstrut\\
 \hline
 \end{tabular}
 \caption{Consecutive patterns of length 3 with repeated letters.}\label{tab2}
\end{table}

Another main result of this paper is the classification of consecutive patterns of length 3 and 4 according to the generalized Wilf equivalence classes from Definition~\ref{def:Wilf}. The equivalence classes for length 3 are given by Tables~\ref{tab1} and~\ref{tab2}. Specifically,
the equivalence $\underline{100}\stackrel{ss}{\sim}\underline{110}$ is the only generalized Wilf equivalence among consecutive patterns of length 3; that is, no other two such patterns are even Wilf equivalent.

The next theorem describes the corresponding classes for patterns of length~4. The patterns $p$ are
listed in increasing order of the values of $\left|\bI_n\left(p\right)\right|$, for $n\le10$; that is, from least avoided to most avoided in inversion sequences of length up to 10.
\begin{thm}\label{Equiv4} A complete list of the generalized Wilf equivalences
between consecutive patterns of length $4$ is as follows: \vspace{-6pt}
\begin{multicols}{2}
 \begin{enumerate}[label=(\roman*),itemsep=1ex,leftmargin=1.5cm]
  \item $\underline{0102}\stackrel{ss}{\sim}\underline{0112}$.
  \item $\underline{0021}\stackrel{ss}{\sim}\underline{0121}$.
  \item $\underline{1002}\stackrel{ss}{\sim}\underline{1012}\stackrel{ss}{\sim}\underline{1102}$.
    \item $\underline{0100}\stackrel{ss}{\sim}\underline{0110}$.
      \item $\underline{2013}\stackrel{ss}{\sim}\underline{2103}$.
  \item $\underline{1200}\stackrel{ss}{\sim}\underline{1210} \stackrel{ss}{\sim}\underline{1220}$.
  \item $\underline{0211}\stackrel{ss}{\sim}\underline{0221}$.
  \item $\underline{1000}\stackrel{ss}{\sim}\underline{1110}$.
  \item $\underline{1001}\stackrel{ss}{\sim}\underline{1011}\stackrel{ss}{\sim}\underline{1101}$.
  \item $\underline{2100}\stackrel{ss}{\sim}\underline{2210}$.
  \item $\underline{2001}\stackrel{ss}{\sim}\underline{2011}\stackrel{ss}{\sim}\underline{2101}\stackrel{ss}{\sim}\underline{2201}$.
  \item $\underline{2012}\stackrel{ss}{\sim}\underline{2102}$.
  \item $\underline{2010}\stackrel{ss}{\sim}\underline{2110}\stackrel{ss}{\sim}\underline{2120}$.
 \item $\underline{3012}\stackrel{ss}{\sim}\underline{3102}$.
 \end{enumerate}
 \end{multicols}
\end{thm}

Again, all the Wilf equivalences among consecutive patters of length 4 are super-strong Wilf equivalences.
This leads us to speculate the following, in analogy to Nakamura's conjecture for consecutive patterns in permutations~{\cite[Conjecture~5.6]{Nakamura}}, which remains open.

\begin{conj}\label{ConjNakamura} Two consecutive patterns of length $k$ in inversion sequences are
strongly Wilf equivalent if and only if they are Wilf equivalent.
\end{conj}

It is in fact possible that a stronger version of Conjecture~\ref{ConjNakamura}
holds, namely that all three types of generalized Wilf equivalence for consecutive patterns in inversion sequences (see Definition~\ref{def:Wilf})
coincide. The results in this paper show that this is the case for consecutive patterns
of length at most~$4$.

On the other hand, following Martinez and Savage~\cite{MartinezSavageII}, one can reframe the notion of patterns of length 3
to consider triples of binary relations among the entries in an occurrence. In a followup paper~\cite{AuliElizaldeII}, we study the consecutive analogue of this framework. Perhaps surprisingly, in the setting of consecutive patterns of relations, the notions of Wilf equivalence and strong Wilf equivalence do not coincide in general.

The rest of the paper is organized as follows.
In Section~\ref{sec:conseclength3} we study consecutive patterns of length~3, proving
the results summarized in Tables~\ref{tab1} and~\ref{tab2}.
Section~\ref{sec:conseclength4} is devoted to the classification of
consecutive patterns of length 4 into generalized Wilf equivalence classes, proving Theorem~\ref{Equiv4}. Some of our results apply to consecutive patterns of arbitrary length, and will be given in Proposition~\ref{prop_general_0s}, and Theorems~\ref{thm:generalize_0102_0112} and ~\ref{thm:generalize_2010_2120_2110}.

\section{Consecutive Patterns of Length 3}\label{sec:conseclength3}

In this section we enumerate inversion sequences avoiding each given consecutive pattern of length~3. Tables~\ref{tab_first8_1} and~\ref{tab_first8_2} provide the first 8 terms of the
sequence $\left|\bI_{n}(p)\right|$ for patterns of length~3 without and with repeated letters, respectively.
There are 13 consecutive patterns of length 3, which fall into 12 Wilf equivalence classes, which are also super-strong Wilf equivalence classes.

\begin{table}[h]
\footnotesize
\centering
 \begin{tabular}{ c@{\hskip 0.5in}l@{\hskip 0.5in}l }
 \hline
 Pattern $p$ & $\left|\bI_{n}(p)\right|$ for $1\le n\le 8$
 \Tstrut\Bstrut\\
 \hline
 $\underline{012}$ & 1, 2, 5, 17, 70, 349, 2017, 13358  \Tstrut\\
 $\underline{021}$ & 1, 2, 6, 23, 107, 585, 3671, 25986 \Tstrut\\
 $\underline{102}$ & 1, 2, 6, 22, 96, 492, 2902, 19350 \Tstrut\\
 $\underline{120}$ & 1, 2, 6, 23, 107, 582, 3622, 25369  \Tstrut\\
 $\underline{201}$ & 1, 2, 6, 24, 118, 684, 4548, 34036  \Tstrut\\
 $\underline{210}$ & 1, 2, 6, 24, 118, 684, 4554, 34192  \Tstrut\Bstrut\\
 \hline
 \end{tabular}
\caption{First 8 terms of the sequence $\left|\bI_{n}(p)\right|$ for consecutive patterns $p$ of length 3 with no repeated letters.}\label{tab_first8_1}
\end{table}

\begin{table}[htbp]
\footnotesize
\centering
 \begin{tabular}{ c@{\hskip 0.5in}l@{\hskip 0.5in}p{7.4cm} }
 \hline
 Pattern $p$ & $\left|\bI_{n}(p)\right|$ for $1\le n\le 8$
 \Tstrut\Bstrut\\
 \hline
 $\underline{000}$ & 1, 2, 5, 19, 91, 531, 3641, 28673 \Tstrut\\
 $\underline{001}$ & 1, 2, 4, 11, 42, 210, 1292, 9352 \Tstrut\\
 $\underline{010}$ & 1, 2, 5, 17, 76, 417, 2701, 20199 \Tstrut\\
 $\underline{011}$ & 1, 2, 5, 17, 75, 407, 2621, 19524 \Tstrut\\
 $\underline{100} \stackrel{ss}{\sim} \underline{110}$ & 1, 2, 6, 23, 109, 618, 4098, 31173 \Tstrut\\
 $\underline{101}$ & 1, 2, 6, 23, 109, 619, 4113, 31352 \Tstrut\Bstrut\\
 \hline
 \end{tabular}
 \caption{First 8 terms of the sequence $\left|\bI_{n}(p)\right|$ for consecutive patterns $p$ of length 3 with repeated letters.}\label{tab_first8_2}
\end{table}

\subsection{The pattern $\underline{000}$}

The following result provides a method to compute
$\left|\bI_{n}\left(\underline{000}\right)\right|$ recursively.

\begin{prop}\label{prop_000} The sequence
$\left|\bI_{n}\left(\underline{000}\right)\right|$ satisfies the recurrence
\[
\left|\bI_{n}\left(\underline{000}\right)\right|=(n-1)\left|\bI_{n-1}\left(\underline{000}\right)\right|+(n-2)\left|\bI_{n-2}\left(\underline{000}\right)\right|,
\]
for $n\geq 3$, with initial terms $\left|\bI_{1}\left(\underline{000}\right)\right|=1$
and $\left|\bI_{2}\left(\underline{000}\right)\right|=2$.
\end{prop}

\begin{proof} It is clear that $\left|\bI_{n}\left(\underline{000}\right)\right|=\left|\bI_{n}\right|$,
for $n<3$,
so $\left|\bI_{1}\left(\underline{000}\right)\right|=1$
and $\left|\bI_{2}\left(\underline{000}\right)\right|=2$. For $n\ge3$, let us count the inversion sequences
$e\in\bI_{n}\left(\underline{000}\right)$ by considering two cases.

If $e_{n-1}\neq e_{n}$, then there are $\left|\bI_{n-1}\left(\underline{000}\right)\right|$ many ways
to choose the first $n-1$ entries of $e$ and $n-1$ possibilities for $e_{n}$, since $e_{n}$ may be an element of $\left\{0,1,\ldots,n-1\right\}\setminus \left\{e_{n-1}\right\}$.

If $e_{n-1}=e_{n}$, then it must be that $e_{n-2}\neq e_{n-1}$. So there are
$\left|\bI_{n-2}\left(\underline{000}\right)\right|$ many ways to choose the first $n-2$
entries of $e$ and $n-2$ possibilities for $e_{n-1}=e_{n}$, which can be any element in $\left\{0,1,\ldots,n-2\right\}\setminus \left\{e_{n-2}\right\}$.
Since every $e\in\bI_{n}\left(\underline{000}\right)$ falls into one of these two cases, the stated recurrence follows.
\end{proof}

The following result generalizes Proposition~\ref{prop_000} to patterns of arbitrary length, providing a recurrence
for the number of inversion sequences avoiding the consecutive pattern consisting of $r$ zeros. We denote this pattern by $\underline{0}^{r}$.

\begin{prop}\label{prop_general_0s}
Let $n\geq r\geq 2$. The sequence $\left|\bI_{n}\left(\underline{0}^{r}\right)\right|$
satisfies the recurrence
\[
\left|\bI_{n}\left(\underline{0}^{r}\right)\right|=\sum_{j=1}^{r-1}(n-j)\left|\bI_{n-j}\left(\underline{0}^{r}\right)\right|,
\]
with initial conditions $\left|\bI_{n}\left(\underline{0}^{r}\right)\right|=n!$, for $1\leq n<r$.
\end{prop}

\begin{proof} Let $A_{j}$ denote the subset of $\bI_{n}\left(\underline{0}^{r}\right)$ consisting of sequences $e$ for which
$j$ is the largest integer such that $e_{n-j+1}=e_{n-j+2}=\cdots=e_{n}$. Given that
$\bI_{n}\left(\underline{0}^{r}\right)=\bigsqcup_{j=1}^{r-1}A_{j}$ (which denotes a disjoint union), we see that
$\left|\bI_{n}\left(\underline{0}^{r}\right)\right|=\sum_{j=1}^{r-1}\left|A_{j}\right|$.

Fix $1\leq j\leq r-1$. Let us construct a sequence $e\in A_{j}$. There are $\left|\bI_{n-j}\left(\underline{0}^{r}\right)\right|$ possibilities
for the first $n-j$ entries of $e$. Once these entries are determined, we have $n-j$ choices for the remaining (repeated) entry
$e_{n-j+1}=e_{n-j+2}=\cdots=e_{n}$, which can be any element in $\left\{0,1,\ldots,n-j\right\}\setminus \left\{e_{n-j}\right\}$, since $e_{n-j+1}\leq n-j$. Given that every sequence in $A_{j}$ may be constructed this way and every sequence constructed
in this manner is in $A_{j}$, we deduce that $\left|A_{j}\right|=(n-j)\left|\bI_{n-j}\left(\underline{0}^{r}\right)\right|$, which concludes the proof.
\end{proof}

A derangement is a permutation without fixed points. The next result relates
inversion sequences avoiding the pattern $\underline{000}$ to derangements.

\begin{cor}\label{cor_000}
Let $d_n$ be the number of derangements of length $n$. For $n\geq 1$,
\[
\left|\bI_{n}\left(\underline{000}\right)\right|=\frac{(n+1)!-d_{n+1}}{n}.
\]
\end{cor}

\begin{proof} We proceed by induction on $n$. The cases $n=1,2$ are straightforward to verify, so let $n\geq 3$.
Suppose that $\left|\bI_{j}\left(\underline{000}\right)\right|=\left((j+1)!-d_{j+1}\right)/j$,
for $1\leq j<n$. By Proposition~\ref{prop_000},
\[
\left|\bI_{n}\left(\underline{000}\right)\right|=(n-1)\left|\bI_{n-1}\left(\underline{000}\right)\right|+(n-2)\left|\bI_{n-2}\left(\underline{000}\right)\right|,
\]
so it follows from the inductive hypothesis that
\begin{align*}
\left|\bI_{n}\left(\underline{000}\right)\right| &= (n-1)\frac{n!-d_{n}}{n-1}+(n-2)\frac{(n-1)!-d_{n-1}}{n-2}\\
&= n!-d_{n}+(n-1)!-d_{n-1}\\
&= \frac{(n+1)!-n(d_{n}+d_{n-1})}{n}.
\end{align*}
It is well-known that the number of derangements satisfies the recurrence $d_{n}=(n-1)(d_{n-1}+d_{n-2})$. This implies that
$\left|\bI_{n}\left(\underline{000}\right)\right|=\left((n+1)!-d_{n+1}\right)/n$.
\end{proof}

The sequence $((n+1)!-d_{n+1})/n$ is listed as A052169 in~\cite{OEIS}. It would be interesting to find a direct combinatorial proof of Corollary~\ref{cor_000}, such as a bijection between $\bI_{n}\left(\underline{000}\right)\times[n]$ and the set of permutations in $S_{n+1}$ with at least one fixed point.

\subsection{The patterns $\underline{100}$ and $\underline{110}$}

Next we prove that patterns $\underline{100}$ and $\underline{110}$ are Wilf equivalent.
We begin by finding a recurrence for $\left|\bI_{n,k}\left(\underline{110}\right)\right|$, as defined in Equation~\eqref{eq:Ink}. The idea is to exploit the fact that if $e=e_{1}e_{2}\dots e_{n}\in \bI_{n}\left(\underline{110}\right)$, then
$e_{1}e_{2}\dots e_{n-1}\in \bI_{n-1}\left(\underline{110}\right)$, as we implicitly did in the proof of
Proposition~\ref{prop_000}. Throughout this section, we use the convention that
\begin{equation}\label{eq:I0} \left|\bI_{0}\left(p\right)\right|=1,
\end{equation} for any pattern $p$.
This value also serves as an initial term in most of the recurrences below, even when not explicitly stated. We also use the convention that $\left|\bI_{n,j}\left(p\right)\right|=0$, if $n<0$.

\begin{prop}\label{prop_110}
Let $n\geq 1$ and $0\leq k<n$. The sequence $\left|\bI_{n,k}\left(\underline{110}\right)\right|$ satisfies the recurrence
\[
\left|\bI_{n,k}\left(\underline{110}\right)\right|=\left|\bI_{n-1}\left(\underline{110}\right)\right|-\sum_{j=k+1}^{n-3}\left|\bI_{n-2,j}\left(\underline{110}\right)\right|.
\]
\end{prop}

\begin{proof} It is straightforward to verify that the recurrence holds for $n=1,2$ (the sum on the right-hand side is empty in these cases). Assume $n\geq 3$. For $j\le n-3$, let
\[
A_{n-1,j}=\left\{e\in\bI_{n-1,j}\left(\underline{110}\right):e_{n-2}=j\right\}.
\]
Note that $e\in\bI_{n,k}\left(\underline{110}\right)$
if and only if $e_{1}e_{2}\dots e_{n-1}\in \bI_{n-1}\left(\underline{110}\right)\setminus \bigsqcup_{j=k+1}^{n-3}A_{n-1,j}$ and $e_{n}=k$. Hence,
\[
\left|\bI_{n,k}\left(\underline{110}\right)\right|=\left|\bI_{n-1}\left(\underline{110}\right)\right|-\sum_{j=k+1}^{n-3}\left|A_{n-1,j}\right|.
\]

It remains to determine $\left|A_{n-1,j}\right|$, for $k+1\leq j\leq n-3$. Since $e\in A_{n-1,j}$ if
and only if $e_{1}e_{2}\dots e_{n-2}\in \bI_{n-2,j}\left(\underline{110}\right)$ and $e_{n-1}=j$, we deduce that
$\left|A_{n-1,j}\right|=\left|\bI_{n-2,j}\left(\underline{110}\right)\right|$.
\end{proof}

It is clear by definition that $\left|\bI_{n}\left(p\right)\right|=\sum_{k=0}^{n-1}\left|\bI_{n,k}\left(p\right)\right|$ for any pattern $p$, so
Proposition~\ref{prop_110} allows us to compute $\left|\bI_{n}\left(\underline{110}\right)\right|$ recursively.
We now show that the number of inversion sequences avoiding the pattern $\underline{100}$ satisfies the same recurrence.

\begin{prop}\label{prop_100}
Let $n\geq 1$ and $0\leq k<n$. The sequence $\left|\bI_{n,k}\left(\underline{100}\right)\right|$ satisfies the recurrence
\[
\left|\bI_{n,k}\left(\underline{100}\right)\right|=\left|\bI_{n-1}\left(\underline{100}\right)\right|-\sum_{j=k+1}^{n-3}\left|\bI_{n-2,j}\left(\underline{100}\right)\right|.
\]
\end{prop}

\begin{proof} The result clearly holds for $n=1,2$, so we assume $n\geq 3$. Let
\[
C_{n-1,k}=\left\{e\in\bI_{n-1,k}\left(\underline{100}\right):e_{n-2}>e_{n-1}\right\}.
\]

If $e\in\bI_{n,k}\left(\underline{100}\right)$, then $e_{1}e_{2}\dots e_{n-1}\in\bI_{n-1}\left(\underline{100}\right)$.
Additionally, since $e_{n-2}e_{n-1}e_{n}$ is not an occurrence of $\underline{100}$, it is not possible to
have $e_{n-2}>e_{n-1}=k$. Thus, $e_{1}e_{2}\dots e_{n-1}\in \bI_{n-1}\left(\underline{100}\right)\setminus C_{n-1,k}$.

Conversely, given $e_{1}e_{2}\dots e_{n-1}\in\bI_{n-1}\left(\underline{100}\right)\setminus C_{n-1,k}$, setting $e_{n}=k$ yields
an inversion sequence $e=e_{1}e_{2}\dots e_{n}\in\bI_{n,k}\left(\underline{100}\right)$.
From this bijective correspondence between $\bI_{n,k}\left(\underline{100}\right)$ and $\bI_{n-1}\left(\underline{100}\right)\setminus C_{n-1,k}$, it follows that
\begin{equation}\label{eq:100_1}
\left|\bI_{n,k}\left(\underline{100}\right)\right|=\left|\bI_{n-1}\left(\underline{100}\right)\right|-\left|C_{n-1,k}\right|.
\end{equation}

Let us determine $\left|C_{n-1,k}\right|$. First suppose $0\leq k\leq n-4$. Note that $e_{1}e_{2}\dots e_{n-1}\in C_{n-1,k}$ if and only if
$e_{1}e_{2}\dots e_{n-2}\in \bI_{n-2,j}\left(\underline{100}\right)$ for some $k+1\leq j\leq n-3$ and $e_{n-1}=k$. Hence,
\begin{equation}\label{eq:100_2}
\left|C_{n-1,k}\right|=\sum_{j=k+1}^{n-3}\left|\bI_{n-2,j}\left(\underline{100}\right)\right|.
\end{equation}
In fact, this formula also holds for $k=n-3,n-2$, since in these cases the sum on the right-hand side is empty, and the left-hand side is zero by definition. The result now follows from Equations~\eqref{eq:100_1} and~\eqref{eq:100_2}.
\end{proof}

Using Equation~\eqref{eq:I0}, the next corollary is an immediate consequence of Propositions~\ref{prop_110} and~\ref{prop_100}.

\begin{cor}\label{cor_110_100} The patterns $\underline{110}$ and $\underline{100}$
are Wilf equivalent. Furthermore, if $n\geq 1$ and $0\leq k<n$, then $\left|\bI_{n,k}\left(\underline{110}\right)\right|=\left|\bI_{n,k}\left(\underline{100}\right)\right|$.
\end{cor}

Next we show that, in fact, patterns $\underline{110}$ and $\underline{100}$ are super-strongly Wilf equivalent.

\begin{defin} We say that a consecutive pattern $p$ is {\it non-overlapping\/} if it is impossible
for two occurrences of $p$ in an inversion sequence to overlap in more than one entry. Equivalently, $p=\underline{p_1\dots p_r}$ is non-overlapping if, for all $1<i<r$, the reductions of $p_1\dots p_i$ and $p_{r-i+1}\dots p_r$ do not coincide.
\end{defin}

\begin{exa} The pattern $\underline{110}$ is non-overlapping, but this is not the case for the pattern
$\underline{1100}$. Indeed, $e=0000221100\in\bI_{10}$ is an example of an inversion sequence
with two occurrences of $\underline{1100}$ which overlap in two entries, namely $e_{5}e_{6}e_{7}e_{8}$ and $e_{7}e_{8}e_{9}e_{10}$.
\end{exa}

\begin{lem}\label{lem:sub_110_100} Let $n$ be a positive integer and $S\subseteq [n]$, then
\[
\left|\left\{e\in\bI_{n}:\Em(\underline{110},e)\supseteq S\right\}\right|=\left|\left\{e\in\bI_{n}:\Em(\underline{100},e)\supseteq S\right\}\right|.
\]
\end{lem}
\begin{proof} Consider the map
\[
\Phi_{S}:\left\{e\in\bI_{n}:\Em(\underline{110},e)\supseteq S\right\}\rightarrow\left\{e'\in\bI_{n}:\Em(\underline{100},e')\supseteq S\right\}
\]
that replaces the occurrences of $\underline{110}$ in $e$ in positions of $S$ with occurrences of $\underline{100}$. More specifically, for $e$ in the domain, $\Phi_{S}(e)=e'$ is defined as
\[
  e'_{j} = \begin{cases}
      e_{j+1}, & \text{if } j-1\in S,\\
      e_{j}, & \text{otherwise,}
      \end{cases}
\]
for $1\le j\le n$.
Note that $e'\in\bI_n$ because, if $j-1\in S$, then $e_{j-1}e_j e_{j+1}$ is an occurrence of $\underline{110}$, and so $e'_j=e_{j+1}<e_{j}<j$.

To show that $\Phi_{S}$ is a bijection, let us describe its inverse. For $e'\in\bI_{n}$ with $\Em(\underline{100},e')\supseteq S$, define $\Psi_{S}(e')=e$ by
\[
  e_{j} = \begin{cases}
      e'_{j-1}, & \text{if } j-1\in S,\\
      e'_{j}, & \text{otherwise,}
      \end{cases}
\]
for $1\le j\le n$.
Note that $e\in\bI_n$ because, if $j-1\in S$, then $e'_{j-1}e'_j e'_{j+1}$ is an occurrence of $\underline{100}$, and so  $e_j=e'_{j-1}<j-1<j$.
Since $\underline{110}$ and $\underline{100}$ are non-overlapping, $\Phi_{S}$ and $\Psi_{S}$ are well-defined. It is clear by construction that $\Phi_{S}$ and $\Psi_{S}$ are inverses of each other.
\end{proof}

The following lemma provides the final argument required to show that $\underline{110}\stackrel{ss}{\sim}\underline{100}$.

\begin{lem}\label{lem:inclusion_exlusion} Let $p$ and $p'$ be two consecutive patterns such that
\[
\left|\left\{e\in\bI_{n}:\Em(p,e)\supseteq S\right\}\right|=\left|\left\{e\in\bI_{n}:\Em(p',e)\supseteq S\right\}\right|
\]
for all positive integers $n$ and all $S\subseteq [n]$. Then $p\stackrel{ss}{\sim}p'$.
\end{lem}
\begin{proof} Given $S\subseteq [n]$, let
\[
f_{=}(S)=\left|\left\{e\in\bI_{n}:\Em(p,e)= S\right\}\right|\quad\textnormal{ and }\quad f_{\geq}(S)=\left|\left\{e\in\bI_{n}:\Em(p,e)\supseteq S\right\}\right|.
\]
For any $T\subseteq [n]$, it is clear that
\[
f_{\geq}(T)=\sum_{S\supseteq T}f_{=}(S),
\]
so the Principle of Inclusion-Exclusion (see~{\cite[Theorem~2.1.1]{Stanley}}) implies
that
\begin{equation}\label{eq:inc_exc1}
f_{=}(T)=\sum_{S\supseteq T}(-1)^{|S\setminus T|}f_{\geq}(S).
\end{equation}

Similarly, letting
\[
f'_{=}(S)=\left|\left\{e\in\bI_{n}:\Em(p',e)= S\right\}\right|\quad\textnormal{ and }\quad f'_{\geq}(S)=\left|\left\{e\in\bI_{n}:\Em(p',e)\supseteq S\right\}\right|,
\]
we have
\begin{equation}\label{eq:inc_exc2}
f'_{=}(T)=\sum_{S\supseteq T}(-1)^{|S\setminus T|}f'_{\geq}(S).
\end{equation}
By assumption, $f_{\geq}(S)=f'_{\geq}(S)$ for all $S\subseteq[n]$, so Equations~\eqref{eq:inc_exc1} and~\eqref{eq:inc_exc2}
imply that $f_{=}(T)=f'_{=}(T)$ for all $T\subseteq[n]$. This proves that $p\stackrel{ss}{\sim}p'$.
\end{proof}

The next proposition is an immediate consequence of Lemmas~\ref{lem:sub_110_100} and~\ref{lem:inclusion_exlusion}. A generalization to patterns of arbitrary length will be given in Theorem~\ref{thm:general_extend_110_100_sergi}.

\begin{prop}\label{prop:110ss100} The patterns $\underline{110}$ and $\underline{100}$ are super-strongly Wilf equivalent.
\end{prop}

\subsection{Other patterns of length 3}

Applying the same technique used in the proof of Propositions~\ref{prop_110} and~\ref{prop_100}, we have
obtained recurrences for $\left|\bI_{n,k}\left(p\right)\right|$ for all consecutive patterns $p$ of length 3. In each case, we
consider an inversion sequence $e=e_{1}e_{2}\dots e_{n}\in\bI_{n,k}(p)$ and analyze the possibilities
for $e_{1}e_{2}\dots e_{n-1}\in\bI_{n-1}(p)$. We then determine that $e_{1}e_{2}\dots e_{n-1}$ may be any sequence in $\bI_{n-1}(p)$, with the exception of certain $e_{1}e_{2}\dots e_{n-2}\in\bI_{n-2,j}(p)$, which are not allowed for some values of $j$. For this reason we obtain recurrences for $\left|\bI_{n,k}\left(p\right)\right|$ in terms of $\left|\bI_{n-1}\left(p\right)\right|$ and $\left|\bI_{n-2,j}\left(p\right)\right|$, with $j$ in a certain range.

We now list recurrences obtained for consecutive patterns of length 3 with no repeated letters.
\begin{prop}\label{prop:rest_recur3_dist} Let $n\geq 1$ and $0\leq k<n$. Then the following recurrences hold:
\begin{enumerate}[label=(\alph*),leftmargin=1.5cm, itemsep=3pt]
\item For $k< n-1$,
\[
\left|\bI_{n,k}\left(\underline{021}\right)\right|=\left|\bI_{n-1}\left(\underline{021}\right)\right|-
(n-2-k)\sum_{j= 0}^{k-1}\left|\bI_{n-2,j}\left(\underline{021}\right)\right|,
\]
and $\left|\bI_{n,n-1}\left(\underline{021}\right)\right|=\left|\bI_{n-1}(\underline{021})\right|$.

\item $\left|\bI_{n,k}\left(\underline{102}\right)\right|=\left|\bI_{n-1}\left(\underline{102}\right)\right|-
\sum_{j\geq1}j\left|\bI_{n-2,j}\left(\underline{102}\right)\right|$.
\item $\left|\bI_{n,k}\left(\underline{120}\right)\right|=\left|\bI_{n-1}\left(\underline{120}\right)\right|-\sum_{j>k}(n-2-j)\left|\bI_{n-2,j}\left(\underline{120}\right)\right|$.
\item $\left|\bI_{n,k}\left(\underline{201}\right)\right|=\left|\bI_{n-1}\left(\underline{201}\right)\right|-k
\sum_{j>k}\left|\bI_{n-2,j}\left(\underline{201}\right)\right|$.
\end{enumerate}
\end{prop}

Next, we describe the recurrence obtained for the pattern $\underline{012}$, which contains terms of the form $\left|\bI_{n-3,j}\left(\underline{012}\right)\right|$. In this case, when we consider an inversion sequence $e=e_{1}e_{2}\dots e_{n}\in\bI_{n,k}(\underline{012})$ and analyze the possibilities
 for $e_{1}e_{2}\dots e_{n-1}\in\bI_{n-1}(\underline{012})$, we determine that $e_{1}e_{2}\dots e_{n-1}$ may be any sequence in $\bI_{n-1}(\underline{012})$, with the exception of certain $e_{1}e_{2}\dots e_{n-3}\in\bI_{n-3,j}(\underline{012})$, which are not allowed for some values of $j$. Indeed, $e=e_{1}e_{2}\dots e_{n}\in\bI_{n,k}(\underline{012})$ if and only if:
 \[
e_{1}e_{2}\dots e_{n-1}\in\bI_{n-1}(\underline{012}), \quad e_{n}=k \quad\text{and}\quad e_{n-2}, e_{n-1}\ \text{do not satisfy}\ e_{n-2}<e_{n-1}<k.
\]
Now $e_{1}e_{2}\dots e_{n-1}$ belongs to $\bI_{n-1}(\underline{012})$ and satisfies $e_{n-2}<e_{n-1}<k$ if and only if $e_{1}e_{2}\dots e_{n-3}\in\bI_{n-3, i}(\underline{012})$ for some $i\geq e_{n-2}$, and $e_{n-2}<e_{n-1}<k$. Hence, making the substitutions $j=e_{n-2}$ and $l=e_{n-1}$, we deduce the following result.

\begin{prop}\label{prop:rest_recur3_012} Let $n\geq 4$ and $0\leq k<n$. Then the sequence $\left|\bI_{n,k}\left(\underline{012}\right)\right|$ satisfies the recurrence
\[
\left|\bI_{n,k}\left(\underline{012}\right)\right|=\left|\bI_{n-1}\left(\underline{012}\right)\right|-\sum_{l= 1}^{k-1}\sum_{j= 0}^{l-1}\sum_{i\geq j}\left|\bI_{n-3,i}\left(\underline{012}\right)\right|,
\]
with initial conditions $\left|\bI_{1,0}\left(\underline{012}\right)\right|=\left|\bI_{2,0}\left(\underline{012}\right)\right|=\left|\bI_{2,1}\left(\underline{012}\right)\right|=\left|\bI_{3,2}\left(\underline{012}\right)\right|=1$
and $\left|\bI_{3,0}\left(\underline{012}\right)\right|=\left|\bI_{3,1}\left(\underline{012}\right)\right|=2$.
\end{prop}

With some algebraic manipulations, the triple sum in Proposition~\ref{prop:rest_recur3_012} may be simplified to obtain the following recurrence, which can be computed more efficiently.

\begin{cor}\label{cor:collapse_sum1} Let $n\geq 4$ and $0\leq k<n$. Then
\[
\left|\bI_{n,k}\left(\underline{012}\right)\right|=\left|\bI_{n-1}\left(\underline{012}\right)\right|-\sum_{i=0}^{k-3}(i+1)\left(k-1-\frac{i}{2}\right)\left|\bI_{n-3,i}\left(\underline{012}\right)\right|
-\frac{k(k-1)}{2}\sum_{i=k-2}^{n-4}\left|\bI_{n-3,i}\left(\underline{012}\right)\right|,
\]
with the same initial conditions as in Proposition~\ref{prop:rest_recur3_012}.
\end{cor}

\begin{proof}
Changing the order of the two outermost summations in the triple sum in Proposition~\ref{prop:rest_recur3_012}, we obtain
\[
\sum_{l= 1}^{k-1}\sum_{j= 0}^{l-1}\sum_{i\geq j}\left|\bI_{n-3,i}\left(\underline{012}\right)\right| = \sum_{j= 0}^{k-2}\sum_{l= j}^{k-2}\sum_{i\geq j}\left|\bI_{n-3,i}\left(\underline{012}\right)\right| = \sum_{j=0}^{k-2}(k-j-1)\sum_{i\geq j}\left|\bI_{n-3,i}\left(\underline{012}\right)\right|
\]
Similarly, changing the order of the rightmost double summation above yields
\begin{align*}
\sum_{l= 1}^{k-1}\sum_{j= 0}^{l-1}\sum_{i\geq j}\left|\bI_{n-3,i}\left(\underline{012}\right)\right| &= \sum_{i=0}^{k-3}\sum_{j=0}^{i}(k-j-1)\left|\bI_{n-3,i}\left(\underline{012}\right)\right| + \sum_{i=k-2}^{n-4}\sum_{j=0}^{k-2}(k-j-1)\left|\bI_{n-3,i}\left(\underline{012}\right)\right|\\
&= \sum_{i=0}^{k-3}(i+1)\left(k-1-\frac{i}{2}\right)\left|\bI_{n-3,i}\left(\underline{012}\right)\right|
-\frac{k(k-1)}{2}\sum_{i=k-2}^{n-4}\left|\bI_{n-3,i}\left(\underline{012}\right)\right|.
\end{align*}
The result then follows from Proposition~\ref{prop:rest_recur3_012}.
\end{proof}

The following is the recurrence obtained for the pattern $\underline{210}$ using a similar argument.

\begin{prop}\label{prop:recur_210} Let $n\geq 1$ and $0\leq k<n$. Then  the sequence $\left|\bI_{n,k}\left(\underline{210}\right)\right|$ satisfies the recurrence
\[
\left|\bI_{n,k}\left(\underline{210}\right)\right|=\left|\bI_{n-1}\left(\underline{210}\right)\right|-
\sum_{l=k+1}^{n-4}\sum_{j=l+1}^{n-3}\sum_{i\leq j}\left|\bI_{n-3,i}(\underline{210})\right|.
\]
\end{prop}

As in Corollary~\ref{cor:collapse_sum1}, we may simplify the triple sum in Proposition~\ref{prop:recur_210} to obtain the following recurrence.

\begin{cor}\label{cor:collapse_sum2} Let $n\geq 1$ and $0\leq k<n$. If $1\leq n\leq 4$ or $0\leq n-5<k$, then
$\left|\bI_{n,k}\left(\underline{210}\right)\right|=\left|\bI_{n-1}\left(\underline{210}\right)\right|$. Otherwise (i.e., $0\leq k\leq n-5$),
\begin{multline*}
\left|\bI_{n,k}\left(\underline{210}\right)\right|=\left|\bI_{n-1}\left(\underline{210}\right)\right|
-\frac{1}{2}\bigg[(n-k-4)(n-k-3)\sum_{i=0}^{k+2}\left|\bI_{n-3,i}(\underline{210})\right|\\
+\sum_{i=k+3}^{n-4}(n-i-2)(n+i-2k-5)\left|\bI_{n-3,i}(\underline{210})\right|\bigg].
\end{multline*}
\end{cor}

Similarly, it is possible to write recurrences for all the remaining consecutive patterns of length 3 with repeated letters. We list them next.

\begin{prop}\label{prop:rest_recur3_same} Let $n\geq 1$ and $0\leq k<n$. Then the following recurrences hold:
\begin{enumerate}[label=(\alph*),leftmargin=1.5cm,, itemsep=3pt]
\item $\left|\bI_{n,k}\left(\underline{001}\right)\right|=\left|\bI_{n-1}\left(\underline{001}\right)\right|-
\sum_{j<k}\left|\bI_{n-2,j}\left(\underline{001}\right)\right|$.
\item $\left|\bI_{n,k}\left(\underline{010}\right)\right|=\left|\bI_{n-1}\left(\underline{010}\right)\right|-
\left(n-2-k\right)\left|\bI_{n-2,k}\left(\underline{010}\right)\right|$.
\item For $k< n-1$,
\[
\left|\bI_{n,k}\left(\underline{011}\right)\right|=\left|\bI_{n-1}\left(\underline{011}\right)\right|-
\sum_{j<k}\left|\bI_{n-2,j}\left(\underline{011}\right)\right|,
\]
and $\left|\bI_{n,n-1}\left(\underline{011}\right)\right|=\left|\bI_{n-1}(\underline{011})\right|$.
\item $\left|\bI_{n,k}\left(\underline{101}\right)\right|=\left|\bI_{n-1}\left(\underline{101}\right)\right|-
k\left|\bI_{n-2,k}\left(\underline{101}\right)\right|$.
\end{enumerate}
\end{prop}

\subsection{From patterns in inversion sequences to patterns in permutations}

In this subsection we discuss some correspondences between consecutive patterns of length 3 in inversion
sequences and permutation patterns. We exploit the bijection $\Theta:S_{n}\rightarrow\bI_{n}$ defined by (\ref{eq:theta_bijection}).
For $\pi\in S_n$ and $1\leq i\leq n$, let
\[
E_{i}(\pi)=\left\{j\in[n]:j<i\textnormal{ and }\pi_{j}>\pi_{i}\right\}.
\]
The entries of $e=\Theta(\pi)$ are given by $e_{i}=\left|E_{i}(\pi)\right|$.

\begin{lem}\label{lem:ref_suggest} Let $\pi\in S_{n}$, and let $e=\Theta(\pi)$ be its corresponding inversion sequence. Then
$\pi_{i}>\pi_{i+1}$ if and only if $e_{i}<e_{i+1}$.
\end{lem}

\begin{proof} Suppose $\pi_{i}>\pi_{i+1}$. Then every $j<i$ such that $\pi_{j}>\pi_{i}$ satisfies
$\pi_{j}>\pi_{i+1}$. It follows that $E_{i}(\pi)\subseteq E_{i+1}(\pi)$. Since $i+1\in E_{i+1}(\pi)\setminus E_{i}(\pi)$, this inclusion is strict, and so $e_{i}<e_{i+1}$.

Conversely, suppose $e_{i}<e_{i+1}$. If $\pi_{i}<\pi_{i+1}$, then $i\not\in E_{i+1}(\pi)$, so there must
exist some $j\in E_{i+1}(\pi)\setminus E_{i}(\pi)$ with $j<i$. But then $\pi_{i+1}<\pi_{j}<\pi_{i}$, which contradicts $\pi_{i}<\pi_{i+1}$. It follows that $\pi_{i}>\pi_{i+1}$.
\end{proof}

\begin{prop}\label{prop_012_permut} Let $\pi\in S_{n}$, and let $e=\Theta(\pi)$ be its corresponding inversion sequence. Then $\pi$ avoids
$\underline{321}$ if and only $e$ avoids $\underline{012}$. Hence,
$\left|\bI_{n}\left(\underline{012}\right)\right|=\left|S_{n}\left(\underline{321}\right)\right|$, for all $n$.
\end{prop}

\begin{proof}
By Lemma~\ref{lem:ref_suggest}, $\pi$ contains $\underline{321}$ if and only if $e$ contains $\underline{012}$.
Equivalently, $\pi$ avoids $\underline{321}$ if and only if $e$ avoids $\underline{012}$.
Consequently, the map $\Theta$ induces a bijection between $S_{n}\left(\underline{321}\right)$ and $\bI_{n}\left(\underline{012}\right)$.
\end{proof}

The sequence $\left|S_{n}\left(\underline{321}\right)\right|$ appears as A049774 in~\cite{OEIS}. The exponential generating function for permutations avoiding $\underline{321}$ is known, see~\cite{DavidBarton,ElizaldeNoy}. It follows from~Proposition~\ref{prop_012_permut} and \cite[Theorem 4.1]{ElizaldeNoy} that
\begin{displaymath}
\sum_{n\ge0} \left|\bI_{n}\left(\underline{012}\right)\right| \frac{z^n}{n!}=\frac{\sqrt{3}}{2}\frac{\exp\left(z/2\right)}{\cos\left(\sqrt{3}z/2+\pi/6\right)}.
\end{displaymath}
Nevertheless, the recurrence in Proposition~\ref{prop:rest_recur3_012} appears to be new.

By Lemma~\ref{lem:ref_suggest}, $\pi$ contains $\underline{\left(r+1\right)r\ldots 1}$ if and only if $e$
contains $\underline{01\ldots r}$. Thus, Proposition~\ref{prop_012_permut} is generalized as follows.

\begin{prop}\label{gen_prop_012_permut} Let $\pi\in S_{n}$, and let $e=\Theta(\pi)$ be its corresponding inversion sequence. Then $\pi$ avoids $\underline{\left(r+1\right)r\ldots 1}$ if and only if $e$ avoids $\underline{01\ldots r}$. Hence,
$\left|\bI_{n}\left(\underline{01\ldots r}\right)\right|=\left|S_{n}\left(\underline{\left(r+1\right)r\ldots 1}\right)\right|$, for all $n$.
\end{prop}

The next result relates $\underline{120}$ and $\underline{021}$ to vincular patterns in permutations.

\begin{prop}\label{prop:pat_vincular} Let $\pi\in S_{n}$, and let $e=\Theta(\pi)$ be its corresponding inversion sequence. Then:
\begin{enumerate}[label=(\alph*),leftmargin=1.5cm,, itemsep=3pt]
\item $\pi$ avoids $3\underline{214}$ if and only if $e$ avoids $\underline{120}$;
\item $\pi$ avoids  $2\underline{413}$ if and only if $e$ avoids $\underline{021}$.
\end{enumerate}
In particular,
$\left|\bI_{n}\left(\underline{120}\right)\right|=\left|S_{n}\left(3\underline{214}\right)\right|$ and $\left|\bI_{n}\left(\underline{021}\right)\right|=\left|S_{n}\left(2\underline{413}\right)\right|$ for all $n$.
\end{prop}

\begin{proof} The proof of both results is quite similar, so we only prove (a).
It suffices to show
that $\pi$ contains $3\underline{214}$ if and only $e$ contains $\underline{120}$.

Suppose that $\pi$ contains an occurrence $\pi_j\pi_i\pi_{i+1}\pi_{i+2}$ of $3\underline{214}$, where $j<i$ and
$\pi_{i+1}<\pi_{i}<\pi_{j}<\pi_{i+2}$. By Lemma~\ref{lem:ref_suggest}, we have
$e_{i}<e_{i+1}$.
If $l\in E_{i+2}(\pi)$, then $l<i+2$ and $\pi_{l}>\pi_{i+2}$. Since $\pi_{i},\pi_{i+1}<\pi_{i+2}$,
it must be that $l<i$, and so $l\in E_{i}(\pi)$. We deduce
that $E_{i+2}(\pi)\subseteq E_{i}(\pi)$. In fact, since $j\in E_{i}(\pi)\setminus E_{i+2}(\pi)$,
the inclusion is strict, so $e_{i+2}<e_{i}$. It follows that $e_{i}e_{i+1}e_{i+2}$ is an occurrence of $\underline{120}$.

Conversely, suppose that $e$ contains an occurrence $e_ie_{i+1}e_{i+2}$ of $\underline{120}$, where $e_{i+2}<e_{i}<e_{i+1}$.
By Lemma~\ref{lem:ref_suggest}, we have $\pi_{i}>\pi_{i+1}$.
If $\pi_i>\pi_{i+2}$, then $E_{i}(\pi)\subseteq E_{i+2}(\pi)$, contradicting that $e_i>e_{i+2}$. It follows that $\pi_{i+2}>\pi_{i}$, and so $\pi_{i}\pi_{i+1}\pi_{i+2}$ is an occurrence of $\underline{213}$. Finally, since $e_i>e_{i+2}$, there exists $j\in E_{i}(\pi)\setminus E_{i+2}(\pi)$. In particular, $j<i$ and $\pi_{i}<\pi_{j}<\pi_{i+2}$, and thus $\pi_j\pi_i\pi_{i+1}\pi_{i+2}$ is an occurrence of
$3\underline{214}$.
\end{proof}

Permutations avoiding the vincular pattern $\underline{143}2$ were studied by
Baxter and Pudwell~\cite{BaxterPudwell}. The sequence $\left|S_{n}\left(\underline{143}2\right)\right|$ appears as A200404 in~\cite{OEIS}, but no
enumerative results seem to be known.
 The next corollary shows that $\left|\bI_{n}\left(\underline{120}\right)\right|=\left|S_{n}\left(\underline{143}2\right)\right|$, and so we can use the recurrence in Proposition~\ref{prop:rest_recur3_dist}(c) to compute these numbers.
Given a permutation $\pi=\pi_{1}\pi_{2}\dots \pi_{n}\in S_{n}$, its reverse-complement is the permutation $\pi^{RC}$, where $(\pi^{RC})_{i}=n+1-\pi_{n+1-i}$, for $1\leq i\leq n$.

\begin{cor}\label{cor:143-2} There is a bijection between $\bI_{n}\left(\underline{120}\right)$ and
$S_{n}\left(\underline{143}2\right)$.
\end{cor}
\begin{proof} The map $\pi\rightarrow \pi^{RC}$ induces a bijection between $S_{n}\left(3\underline{214}\right)$ and $S_{n}\left(\underline{143}2\right)$. Indeed, for $1\leq j<i\leq n$, we have
$\pi_{i+1}<\pi_{i}<\pi_{j}<\pi_{i+2}$ if and only if $n+1-\pi_{i+2}<n+1-\pi_{j}<n+1-\pi_i<n+1-\pi_{i+1}$,
which is equivalent to $\pi^{RC}_{n-i-1}<\pi^{RC}_{n-j+1}<\pi^{RC}_{n-i+1}<\pi^{RC}_{n-i}$. Thus,
$\pi$ contains $3\underline{214}$ if and only if $\pi^{RC}$ contains $\underline{143}2$.
The result now follows from Proposition~\ref{prop:pat_vincular}(a).
\end{proof}

Next we show that the sequence counting $\underline{021}$-avoiding inversion sequences coincides with that counting $\underline{132}4$-avoiding permutations, which is given by A071075 in~\cite{OEIS}.

\begin{cor}\label{cor:132-4} For all $n$, $\left|\bI_{n}\left(\underline{021}\right)\right|=\left|S_{n}\left(\underline{132}4\right)\right|$.
\end{cor}

\begin{proof} As in proof of Corollary~\ref{cor:143-2}, the map $\pi\rightarrow \pi^{RC}$ induces a bijection between $S_{n}\left(2\underline{413}\right)$ and $S_{n}\left(\underline{241}3\right)$. Similarly, by applying the reversal operation $\pi_1\pi_2\dots\pi_n\rightarrow \pi_n\dots\pi_2\pi_1$, we observe that $\left|S_{n}\left(132\right)\right|=\left|S_{n}\left(231\right)\right|$. It follows from~\cite[Proposition 3.1]{ElizaldeII} or~{\cite[Corollary~15]{KitaevII}} that $\left|S_{n}\left(\underline{132}4\right)\right|=\left|S_{n}\left(\underline{231}4\right)\right|$.
Finally, a result of Baxter and Shattuck~\cite{BaxterShattuck} states that $\left|S_{n}\left(\underline{241}3\right)\right|=\left|S_{n}\left(\underline{231}4\right)\right|$. Combining these with Proposition~\ref{prop:pat_vincular}(b), we obtain
\[
\left|\bI_{n}\left(\underline{021}\right)\right|=\left|S_{n}\left(2\underline{413}\right)\right|=\left|S_{n}\left(\underline{241}3\right)\right|=\left|S_{n}\left(\underline{231}4\right)\right|=\left|S_{n}\left(\underline{132}4\right)\right|.
\qedhere
\]
\end{proof}

The exponential generating function for permutations avoiding $\underline{132}4$ can be deduced by combining \cite[Proposition 3.1]{ElizaldeII} and \cite[Theorem 4.1]{ElizaldeNoy}. It follows from Corollary~\ref{cor:132-4} that
\begin{displaymath}\sum_{n\ge0} \left|\bI_{n}\left(\underline{021}\right)\right| \frac{z^n}{n!}=\exp\left(\int_{0}^{z} \frac{dt}{1-\int_0^t e^{-u^2/2} du}\right) .\end{displaymath}
A functional equation satisfied by the ordinary generating function for $\underline{132}4$-avoiding permutations is given by Baxter and Shattuck~\cite{BaxterShattuck}. The recurrence in Proposition~\ref{prop:rest_recur3_dist}(a) appears to be new.

\section{Consecutive Patterns of Length 4}\label{sec:conseclength4}

In this section we study consecutive patterns of length 4 in inversion sequences. Specifically, we determine the Wilf equivalence classes for consecutive patterns of length 4. This classification is summarized in Theorem~\ref{Equiv4}. There are a total of 75 consecutive  patterns of length 4, which fall into 55 Wilf equivalence classes. This is also the number of both strong and super-strong Wilf equivalence classes because, as mentioned before, all the equivalences among consecutive patterns of length 4 are super-strong equivalences. Of these 55 classes, 1 contains four patterns, 4 contain three patterns, 9 contain two patterns, and 41 contain one pattern.

The only consecutive pattern $p$ of length 4 for which the sequence $\left|\bI_{n}(p)\right|$ appears in the OEIS~\cite{OEIS} is $p=\underline{0123}$. We have that $\left|\bI_{n}(\underline{0123})\right|=\left|S_{n}(\underline{1234})\right|$, which is sequence A117158 in~\cite{OEIS}.
By Proposition~\ref{gen_prop_012_permut}, the map $\Theta$ gives a bijection between $\bI_{n}(\underline{0123})$ and $S_{n}(\underline{4321})$. The exponential generating function for this sequence is known~\cite{DavidBarton,ElizaldeNoy}. It follows from
\cite[Theorem 4.3]{ElizaldeNoy} that
\begin{displaymath}\sum_{n\ge0} \left|\bI_{n}\left(\underline{0123}\right)\right| \frac{z^n}{n!}=\frac{2}{\cos z-\sin z+e^{-z}}.\end{displaymath}

\subsection{Equivalences proved via bijections}\label{subsec:equiv4_biject}

In this subsection we use bijections to prove equivalences (ii)--(vii), (ix), (xi), (xii) and (xiv) in Theorem~\ref{Equiv4}.
All the equivalences established in this subsection are in fact stronger than super-strong Wilf equivalence: we show not only that $p\stackrel{ss}{\sim} p'$, but rather that
\begin{multline}\label{eq:symss}
\left|\left\{e\in\bI_{n}:\Em(p,e)=S,\,\Em(p',e)=T\right\}\right|=\left|\left\{e\in\bI_{n}:\Em(p,e)=T,\,\Em(p',e)=S\right\}\right|,
\end{multline}
for all $n$ and all $S,T\subseteq [n]$. In other words, the joint distribution of the positions of the occurrences of $p$ and $p'$ is symmetric. Taking the union over all $T\subseteq [n]$ in Equation~\eqref{eq:symss}, it clearly implies that $p\stackrel{ss}{\sim} p'$.

Equation~\eqref{eq:symss} is proved by exhibiting a bijection from $\bI_{n}$ to itself that changes all occurrences of $p$ into occurrences of $p'$, and vice versa.
The following notion will help us construct such a bijection.

\begin{defin}
Two consecutive patterns $p$ and $p'$ are {\it mutually non-overlapping\/} if it is impossible
for an occurrence of $p$ and an occurrence of $p'$ in an inversion sequence to overlap in more than one entry; equivalently, if for all $1<i<r$, the reductions of $p_{1}\ldots p_{i}$ and $p'_{r-i+1}\ldots p_{r}$ do not coincide, and neither do the reductions of $p'_{1}\ldots p'_{i}$ and $p_{r-i+1}\ldots p_{r}$
\end{defin}

\begin{exa} The patterns $\underline{110}$ and $\underline{010}$ are mutually non-overlapping. However, the same cannot be said about
the patterns $\underline{110}$ and $\underline{100}$. For example, the inversion sequence $e=0002211\in\bI_{7}$ has occurrences of
$\underline{110}$ and $\underline{100}$ overlapping in two entries, $e_{5}$ and~$e_{6}$.
\end{exa}

Next we prove equivalence (ii) in Theorem~\ref{Equiv4}.

\begin{prop}\label{prop:0021_ss_0121} The patterns $p=\underline{0021}$ and $p'=\underline{0121}$ satisfy Equation~\eqref{eq:symss}. In particular,  $\underline{0021}\stackrel{ss}{\sim}\underline{0121}$.
\end{prop}

\begin{proof} We define a map ${\phi:\bI_{n}\rightarrow\bI_{n}}$ that
changes every occurrence of $\underline{0021}$ into an occurrence of $\underline{0121}$, and vice versa. Specifically, let $\phi(e)=e'$, where
\[
  e'_{j} = \begin{cases}
      e_{j+2}, & \text{if }e_{j-1}e_{j}e_{j+1}e_{j+2} \text{ is an occurrence of } \underline{0021},\\
      e_{j-1}, & \text{if }e_{j-1}e_{j}e_{j+1}e_{j+2} \text{ is an occurrence of } \underline{0121},\\
      e_{j}, & \text{otherwise.}
      \end{cases}
\]
For instance, if $e=00032454$, then $\phi(e)=e'=00232254$.

First, let us show that $e'$ is an inversion sequence by checking that $e'_j<j$ for all $j$.
Indeed, if ${e_{j-1}e_{j}e_{j+1}e_{j+2}}$ is an occurrence of $\underline{0121}$, then $e'_{j}=e_{j-1}<j-1<j$.
Additionally, if ${e_{j-1}e_{j}e_{j+1}e_{j+2}}$ is an occurrence of $\underline{0021}$, then $e'_{j}=e_{j+2}<e_{j+1}<j+1$, so $e'_{j}<j$.

Since $\underline{0021}$ and $\underline{0121}$ are non-overlapping and mutually non-overlapping, occurrences of these patterns can only overlap in one entry. Since $\phi$ switches occurrences of $\underline{0021}$ and $\underline{0121}$ and does not change the first or the last entry of any occurrence, we conclude that $\phi^{-1}=\phi$, so $\phi$ is a bijection.
Furthermore, the positions of the occurrences of $\underline{0021}$ (respectively, $\underline{0121}$) in $e\in\bI_n$ are precisely the positions of the occurrences of $\underline{0121}$ (respectively, $\underline{0021}$) in $\phi(e)$. It follows that Equation~\eqref{eq:symss} holds and that $\underline{0021}\stackrel{ss}{\sim}\underline{0121}$.
\end{proof}

The idea in the proof of Proposition~\ref{prop:0021_ss_0121} can be exploited to prove a more general result. In order to do this, we need to introduce some terminology.

\begin{defin}\label{def:valid_change} Let $p=\underline{p_{1}p_{2}\ldots p_{r}}$ and $p'=\underline{p'_{1}p'_{2}\ldots p'_{r}}$ be two consecutive patterns such that
$p_1=p'_1$, $p_r=p'_r$, and $\max_i\{p_i\}=\max_i\{p'_i\}=d$ (that is, $p$ and $p'$ have the same number $d+1$ of distinct entries).
Suppose that $e\in\bI_{n}$ has an occurrence of $p$ in position $i$, and let $f:\{0,1,\dots,d\}\to\{e_{i},e_{i+1},\ldots,e_{i+r-1}\}$ (where the right-hand side is viewed as a set, disregarding repetitions) be the order-preserving bijection such that ${e_{i}e_{i+1}\ldots e_{i+r-1}}=f(p_1)f(p_2)\ldots f(p_r)$.

Define a sequence $e'=e'_1e'_2\ldots e'_n$ by setting
\begin{displaymath}{e'_{i}e'_{i+1}\ldots e'_{i+r-1}}=f(p'_1)f(p'_2)\ldots f(p'_r),\end{displaymath}
and $e'_{j}=e_{j}$ for $j<i$ or $j>i+r-1$. This operation is called a {\em change} of the occurrence ${e_{i}e_{i+1}\ldots e_{i+r-1}}$ of $p$ into an occurrence of $p'$. If, additionally, $e'\in\bI_n$ (that is, $e'_j<j$ for all $j$), then we say that this change is {\em valid}.
\end{defin}

\begin{exa}\label{exa:changing} Consider $e=00134015331\in\bI_{11}$. It is not valid to change the occurrence ${e_{2}e_{3}e_{4}e_{5}e_{6}}$ of $\underline{01230}$ into an occurrence of $\underline{03210}$. Indeed, such a change would yield $e'=00431015331$, which is not an inversion sequence.
On the other hand, it is valid to change the occurrence ${e_{7}e_{8}e_{9}e_{10}e_{11}}$ of $\underline{02110}$ into an occurrence of $\underline{02100}$, and this change yields $e'=00134015311\in\bI_{11}$.
\end{exa}

\begin{defin}\label{def:changeable} Let $p=\underline{p_{1}p_{2}\ldots p_{r}}$ and $p'=\underline{p'_{1}p'_{2}\ldots p'_{r}}$ be two consecutive patterns such that
\[
p_{1}=p'_{1}, \quad p_{r}=p'_{r}, \quad\text{and}\quad \max_i\{p_i\}=\max_i\{p'_i\}.
\]
We say that $p$ is {\em changeable} for $p'$ if, for all $1\leq i\leq r$,
\[
p'_{i}\leq \max\left(  \{p_j : 1\le j\le i\} \cup \{p_j - j + i : i<j\leq r\} \right).
\]
If $p$ is changeable for $p'$ and $p'$ is changeable for $p$, then we say that $p$ and $p'$ are {\em interchangeable}.
\end{defin}

\begin{exa}\label{exa:change_patterns_with_def}  Consider the consecutive patterns $p=\underline{01230}$ and $p'=\underline{03210}$. Note that
\[
p'_{2}=3>1= \max\left(  \{p_j : 1\le j\le 2\} \cup \{p_j - j + 2 : 2<j\leq 5\} \right).
\]
Hence, $\underline{01230}$ is not changeable for $\underline{03210}$.

On the other hand, the patterns $p=\underline{0021}$ and $p'=\underline{0121}$ are interchangeable. Indeed, in this case
$p'_{2}=1= \max\left(  \{p_j : 1\le j\le 2\} \cup \{p_j - j + 2 : 2<j\leq 4\} \right)$ and $p'_{i}=p_{i}$ for $i\neq 2$, so $p$ is changeable for $p'$. Similarly, $p'$ is changeable for $p$ because $p_{i}\leq p'_{i}$, for all~$i$.
\end{exa}

The next lemma relates Definitions~\ref{def:valid_change} and~\ref{def:changeable}, showing why the concept of changeability is useful.

\begin{lem}\label{lem:characterize_changeable_patterns} Let $p=\underline{p_{1}p_{2}\ldots p_{r}}$ and $p'=\underline{p'_{1}p'_{2}\ldots p'_{r}}$ be two consecutive patterns. Then $p$ is changeable for $p'$ if and only if, for every inversion sequence $e\in\bI_{n}$, it is valid to change any given occurrence of $p$ in $e$ into an occurrence of $p'$.
\end{lem}
\begin{proof}
First we prove the forward direction. Suppose that $p$ is changeable for $p'$. Let $e\in\bI_{n}$ be an inversion sequence with an occurrence of $p$ in position $i$. Let $e'$ be the sequence obtained by changing the occurrence ${e_{i}e_{i+1}\ldots e_{i+r-1}}$ of $p$ into an occurrence of $p'$. To prove that this change is valid, it suffices to show that $e'_{i+t-1}<i+t-1$ for all $t$ such that $p'_{t}\neq p_{t}$.

Let $t$ be such that $p'_{t}\neq p_{t}$, and let
$M_{t}=\max\left(  \{p_j : 1\le j\le t\} \cup \{p_j - j + t : t<j\leq r\} \right)$. By assumption, $p'_t\le M_t$.
If $M_{t}=p_{j}$ for some $1\leq j\leq t$, then $p'_{t}\leq p_{j}$ implies that $e'_{i+t-1}\leq e_{i+j-1}<i+j-1\leq i+t-1$ because $e'_{i+t-1}=f(p'_{t})$ and $e_{i+j-1}=f(p_{j})$, where $f$ is the order-preserving bijection $f$ in Definition~\ref{def:valid_change}. On the other hand, if $M_{t}=p_{j}-j+t$ for some $t<j\leq r$, then $p'_{t}\leq p_{j}-j+t$ implies that $e'_{i+t-1}\leq e_{i+j-1}-j+t<(i+j-1)-j+t=i+t-1$.

To prove the converse, suppose now that $p$ is not changeable for $p'$. We construct an inversion sequence $e$ with an occurrence of $p$ whose change into an occurrence of $p'$ is not valid. Let $t$ be such that
\begin{equation}\label{eq:proof_changeable_charact}
p'_{t}> \max\left(  \{p_j : 1\le j\le t\} \cup \{p_j - j + t : t<j\leq r\} \right).
\end{equation}

Let $s$ be the smallest integer such that $c=0^{s}p$ is an inversion sequence and $s+t-p'_{t}$ is nonnegative. Equivalently,
\[
s=\max\left( \{p_j-j+1: 1\le j\le r\} \cup \{p'_{t}-t\} \right).
\]
We write $c=c_{1}c_{2}\ldots c_{s+r}$, where $c_{i}=0$ for $1\leq i\leq s$ and $c_{s+j}=p_{j}$ for $1\leq j\leq r$. Now, define $e=e_{1}e_{2}\ldots e_{s+r}$ by setting
\[
  e_{i} = \begin{cases}
      c_{i}+s+t-p'_{t}, & \text{if } c_{i}\geq p'_{t},\\
      c_{i}, & \text{otherwise}.
      \end{cases}
\]
Next we show that $e$ is an inversion sequence. If $i$ is such that $e_{i}=c_{i}$, then $e_{i}<i$, because $c_{i}$ is an inversion sequence. If $i$ is such that $e_{i}\neq c_{i}$, then $i=s+j$ for some $1\leq j\leq r$, and $p'_{t}\leq c_{i}=c_{s+j}=p_{j}$. By inequality~\eqref{eq:proof_changeable_charact}, $p'_{t}\le p_{j}$ implies that $t<j\le r$, and so $p'_{t}>p_{j}-j+t$. Hence,
\[
e_{i}=c_{i}+s+t-p'_{t}=p_{j}+s+t-p'_{t}<p_{j}+s+t-(p_{j}-j+t)=s+j=i.
\]
We have shown that $e_{i}<i$ in any case, and so $e\in\bI_{s+r}$.

Since $c$ has an occurrence of $p$ in position $s+1$, so does $e$. Indeed, by our definition of $s$, we know that $s+t-p'_{t}\geq 0$, so adding $s+t-p'_{t}$ to the values $c_{i}$ for which $c_{i}\geq p'_{t}$ does not alter their relative order. Let $e'$ be the sequence obtained by changing the occurrence $e_{s+1}e_{s+2}\ldots e_{s+r}$ of $p$ into an occurrence of $p'$. Then $e'_{s+t}=p'_t+s+t-p'_{t}=s+t$. This implies that $e'\not\in\bI_{s+r}$, so the change is not valid.
\end{proof}

\begin{exa} We showed in Example~\ref{exa:change_patterns_with_def} that $\underline{01230}$ is not changeable for $\underline{03210}$.
By Lemma~\ref{lem:characterize_changeable_patterns}, this is equivalent to the existence of some inversion sequence with an occurrence of $\underline{01230}$ whose change into $p'$ is not valid. Example~\ref{exa:changing} gives such an inversion sequence.

On the other hand, we determined in Example~\ref{exa:change_patterns_with_def} that $\underline{0021}$ and $\underline{0121}$ are interchangeable.
Again by Lemma~\ref{lem:characterize_changeable_patterns}, this is equivalent to the fact that, in every $e\in\bI_{n}$, it is valid to change any occurrence of $\underline{0021}$ (respectively, $\underline{0121}$) in $e$ into an occurrence of $\underline{0121}$ (respectively, $\underline{0021}$). This property was shown in the proof of Proposition~\ref{prop:0021_ss_0121}.
\end{exa}

\begin{thm}\label{thm:general_extend_110_100_biject} Let $p$ and $p'$ be non-overlapping, mutually non-overlapping and interchangeable consecutive patterns. Then $p$ and $p'$ satisfy Equation~\eqref{eq:symss}. In particular, $p\stackrel{ss}{\sim} p'$.
\end{thm}

\begin{proof} Mimicking the proof of Proposition~\ref{prop:0021_ss_0121}, we define a map $\phi:\bI_{n}\rightarrow\bI_{n}$ that changes occurrences of $p$ into occurrences of $p'$ and vice versa. Since $p$ and $p'$ are interchangeable, Lemma~\ref{lem:characterize_changeable_patterns} implies that it is possible to make these changes. Moreover, since $p$ and $p'$ are non-overlapping and mutually non-overlapping, all of these changes can be made simultaneously, and therefore $\phi$ is well-defined and $\phi^{-1}=\phi$. Hence, $\phi$ is a bijection and the result follows.
\end{proof}

We now list equivalences between consecutive patterns of length 4 that follow as particular instances of Theorem~\ref{thm:general_extend_110_100_biject}. In each case, it can be easily verified that the patterns in question are non-overlapping, mutually non-overlapping and interchangeable. We use the same labeling from Theorem~\ref{Equiv4}.

\begin{cor}\label{cor:biject} The following equivalences hold:\vspace{-6pt}
\begin{multicols}{2}
 \begin{enumerate}[itemsep=1ex,leftmargin=1.5cm]
 \item[(iii)] $\underline{1002}\stackrel{ss}{\sim}\underline{1012}\stackrel{ss}{\sim}\underline{1102}$.
\item[(iv)] $\underline{0100}\stackrel{ss}{\sim}\underline{0110}$.
\item[(v)] $\underline{2013}\stackrel{ss}{\sim}\underline{2103}$.
\item[(vi)] $\underline{1200}\stackrel{ss}{\sim}\underline{1210}
\stackrel{ss}{\sim}\underline{1220}$.
\item[(vii)] $\underline{0211}\stackrel{ss}{\sim}\underline{0221}$.
\item[(ix)] $\underline{1001}
\stackrel{ss}{\sim}\underline{1011}\stackrel{ss}{\sim}\underline{1101}$.
\item[(xi)] $\underline{2001}\stackrel{ss}{\sim}\underline{2011}\stackrel{ss}{\sim}\underline{2101}
\stackrel{ss}{\sim}\underline{2201}$.
\item[(xii)] $\underline{2012}\stackrel{ss}{\sim}\underline{2102}$.
\item[(xiv)] $\underline{3012}\stackrel{ss}{\sim}\underline{3102}$.
 \end{enumerate}
 \end{multicols}
\end{cor}

\subsection{Equivalences between non-overlapping patterns proved via inclusion-exclusion}\label{subsec:equiv4_non_over_inc_exc}

The equivalences between consecutive patterns of length 4 in this subsection follow from the next result, which has weaker hypotheses than Theorem~\ref{thm:general_extend_110_100_biject}. It relaxes the condition  of $p$ and $p'$ being mutually non-overlapping, at the expense of not proving Equation~\eqref{eq:symss} and not producing a direct bijection changing all the occurrences of $p$ into occurrences of $p'$.

\begin{thm}\label{thm:general_extend_110_100_sergi} Let $p$ and $p'$ be non-overlapping and interchangeable consecutive patterns. Then $p\stackrel{ss}{\sim} p'$.
\end{thm}
\begin{proof} We mimic the proof of Lemma~\ref{lem:sub_110_100} and prove that
\[
\left|\left\{e\in\bI_{n}:\Em(p,e)\supseteq S\right\}\right|=\left|\left\{e\in\bI_{n}:\Em(p',e)\supseteq S\right\}\right|
\]
for all positive integers $n$ and all $S\subseteq [n]$.

Fix $n$ and $S$, and consider the map
\[
\Phi_{S}:\left\{e\in\bI_{n}:\Em(p,e)\supseteq S\right\}\rightarrow\left\{e\in\bI_{n}:\Em(p',e)\supseteq S\right\}
\]
that changes the occurrences of $p$ in $e$ in positions of $S$ into occurrences of $p'$. Given that $p$ is non-overlapping and changeable for $p'$, Lemma~\ref{lem:characterize_changeable_patterns} implies that $\Phi_{S}$ is well-defined.

Similarly, we define the map
$\Psi_{S}:\left\{e\in\bI_{n}:\Em(p',e)\supseteq S\right\}\rightarrow\left\{e\in\bI_{n}:\Em(p,e)\supseteq S\right\}$ that changes the occurrences of $p'$ in $e$ in positions of $S$ into occurrences of $p$. Since $p'$ is non-overlapping and changeable for $p$, we deduce again that $\Psi_{S}$ is well-defined. It is clear that $\Phi_{S}$ and $\Psi_{S}$ are inverses of each other, which implies that $\Phi_{S}$ is a bijection.
The result now follows from Lemma~\ref{lem:inclusion_exlusion}.
\end{proof}

The next corollary is a consequence of Theorem~\ref{thm:general_extend_110_100_sergi}. The equivalences are labeled as in Theorem~\ref{Equiv4}.

\begin{cor}\label{cor:extend_110_100_sergi} The following equivalences hold:\vspace{-6pt}
\begin{multicols}{2}
\begin{enumerate}[label=(\roman*),leftmargin=1.5cm,, itemsep=3pt]
\item[(viii)] $\underline{1000}\stackrel{ss}{\sim}\underline{1110}$.
\item[(x)] $\underline{2100}\stackrel{ss}{\sim}\underline{2210}$.
  \end{enumerate}
   \end{multicols}
\end{cor}

\begin{proof} It is straightforward to verify that $\underline{1000}$ and $\underline{1110}$ are non-overlapping and interchangeable, and that the same is true for $\underline{2100}$ and $\underline{2210}$. The result follows from Theorem~\ref{thm:general_extend_110_100_sergi}.
\end{proof}

The patterns $\underline{1000}$ and $\underline{1110}$ (respectively, $\underline{2100}$ and $\underline{2210}$) may overlap with one another, so it is not possible to deduce Corollary~\ref{cor:extend_110_100_sergi} from Theorem~\ref{thm:general_extend_110_100_biject}.

\subsection{Equivalences between overlapping patterns proved via inclusion-exclusion}\label{subsec:equiv4_over_inc_exc}

All the equivalences proved so far have relied on the patterns being non-overlapping. In this subsection we adapt the proof of Theorem~\ref{thm:general_extend_110_100_sergi} to patterns that may have occurrences overlapping in more than one entry. We start by proving equivalence (i) from Theorem~\ref{Equiv4}.

\begin{prop}\label{prop:0102_0112} The patterns $\underline{0102}$ and $\underline{0112}$ are super-strongly Wilf equivalent.
\end{prop}

\begin{proof} We will prove that, for all $n$ and all $S\subseteq[n]$,
\begin{equation}\label{eq:0102_0112}
\left|\left\{e\in\bI_{n}:\Em(\underline{0102},e)\supseteq S\right\}\right|=\left|\left\{e\in\bI_{n}:\Em(\underline{0112},e)\supseteq S\right\}\right|.
\end{equation}
There is no loss of generality in assuming that there are no consecutive elements in $S$, because no two occurrences of $\underline{0102}$ (or $\underline{0112}$) can overlap in three entries.

We can write $S$ uniquely as a disjoint union of {\em blocks}, which we define as maximal subsets whose entries form an arithmetic sequence with difference $2$. Explicitly, $S=\bigsqcup_{j=1}^{m}B_{j}$, with blocks $B_{j}=\left\{i_{j},i_{j}+2,i_{j}+4\ldots,i_{j}+2l_{j}\right\}$ and $i_{j}+2(l_{j}+1)<i_{j+1}$ for $1\leq j\leq m-1$.

Define a map $\Phi_{S}:\left\{e\in\bI_{n}:\Em(\underline{0102},e)\supseteq S\right\}\rightarrow\left\{e\in\bI_{n}:\Em(\underline{0112},e)\supseteq S\right\}$ by setting $\Phi_{S}\left(e\right)=e'$, where
\[
  e'_{i} = \begin{cases}
      e_{i-1}, & \text{if } i-2\in S,\\
      e_{i}, & \text{otherwise}.
      \end{cases}
\]
In other words, $\Phi_{S}$ changes occurrences of $\underline{0102}$ in positions of $S$ into occurrences of $\underline{0112}$. To see that $\Phi_{S}$ is a bijection, we define its inverse map $\Psi_{S}$ by setting $\Psi_{S}(e')=e$, where
\[
  e_{i} = \begin{cases}
      e'_{i_{j}}, & \text{if } i-2\in B_{j},\\\
      e'_{i}, & \text{otherwise}.
      \end{cases}
\]
The map $\Psi_{S}$ changes occurrences of $\underline{0112}$ in positions of $S$ into occurrences of $\underline{0102}$.

Since $\Phi_{S}$ is a bijection, Equation~\eqref{eq:0102_0112} is proved, and the result follows from Lemma~\ref{lem:inclusion_exlusion}.
\end{proof}

A schematic diagram of the maps $\Phi_{S}$ and $\Psi_{S}$ from the above proof appears in Figure~\ref{fig:Scheme_0102_0112}. An example of the computation of $\Phi_{S}$ is given next.

\begin{figure}[htb]
	\noindent \begin{centering}
  \includegraphics[width=0.9\textwidth]{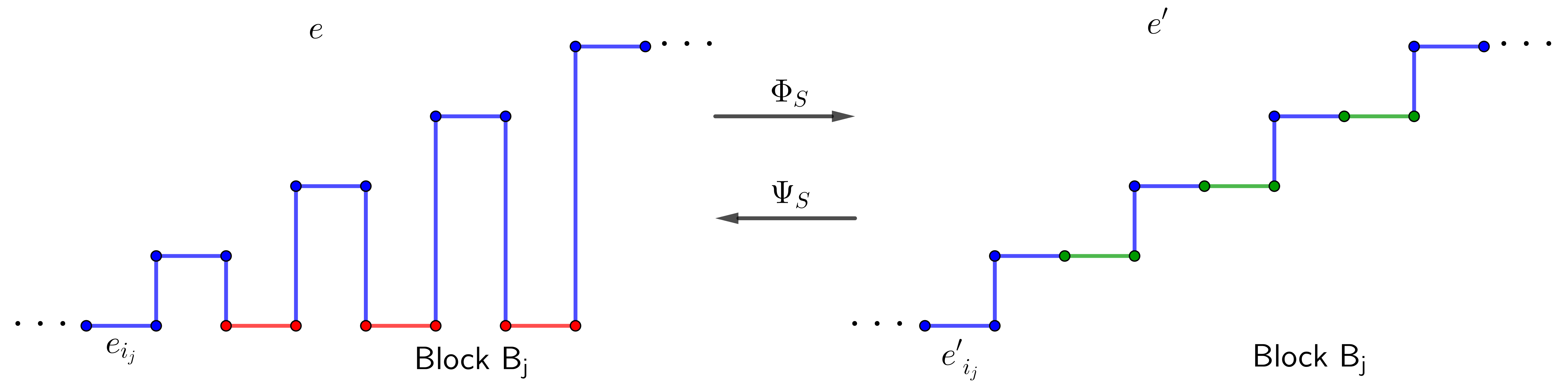}
		\par\end{centering}

	\protect\caption{Schematic diagram of the behavior of $\Phi_{S}$ and $\Psi_{S}$ from Proposition~\ref{prop:0102_0112} on occurrences within a block.}
	\label{fig:Scheme_0102_0112}
\end{figure}

\begin{exa}\label{exa:0102_0112}
\sloppy Let $S=\left\{3,5,9\right\}$ and ${e=0102040523262889}\in\bI_{16}$, which satisfy $S\subseteq\Em(\underline{0102},e)=\left\{1,3,5,9,11\right\}$.
Changing the occurrences of $\underline{0102}$ in positions of $S$ into occurrences of $\underline{0112}$, we obtain $\Phi_{S}(e)=e'=0102244523362889$.
Note that $S\subseteq\Em(\underline{0112},e')=\left\{3,5,9,13\right\}$, and that $\Psi_S(e')=e$.
\end{exa}

Next we prove equivalence (xiii) from Theorem~\ref{Equiv4}.

\begin{prop}\label{prop:2010_2110_2120} The patterns $\underline{2010}$, $\underline{2110}$  and $\underline{2120}$ are super-strongly Wilf equivalent.
\end{prop}

\begin{proof}
For any $n$ and $S\subseteq [n]$, we construct a bijection
\[
\Phi_{S}:\left\{e\in\bI_{n}:\Em(p,e)\supseteq S\right\}\rightarrow\left\{e\in\bI_{n}:\Em(p',e)\supseteq S\right\}
\]
that changes occurrences of $p=\underline{2010}$ in positions of $S$ into occurrences of $p'=\underline{2110}$ (likewise, $p'=\underline{2120}$).
The construction is very similar to the proof of Proposition~\ref{prop:0102_0112}, and so the details are omitted.
Using Lemma~\ref{lem:inclusion_exlusion}, the result follows.
\end{proof}

\begin{proof}[Proof of Theorem~\ref{Equiv4}]
All the listed super-strong Wilf equivalences follow from Propositions~\ref{prop:0021_ss_0121}, \ref{prop:0102_0112} and~\ref{prop:2010_2110_2120}, and Corollaries~\ref{cor:biject} and~\ref{cor:extend_110_100_sergi}. A brute force computation of the first ten terms of the sequences $\left|\bI_{n}\left(p\right)\right|$ for each consecutive pattern $p$ of length 4 shows that there are no additional equivalences, and so the list gives a complete classification.
\end{proof}

\sloppy We can generalize Proposition~\ref{prop:0102_0112} to patterns of the form ${\underline{0^{r}\,1\,0^{r}\,2\,0^{r}\ldots (s-1)\,0^{r}\,s}}$ and ${\underline{0^{r}\,1\,1^{r}\,2\,2^{r}\ldots (s-1)\,(s-1)^{r}\,s}}$, with $r\ge1$ and $s\geq 2$, by proving a statement analogous to \eqref{eq:0102_0112}. Indeed, we generalize the notion of a block to be a maximal subset of $S$ whose elements form an increasing  sequence where consecutive differences are of the form $a(r+1)$, with $1\leq a\leq s-1$. Again, we can write $S$ uniquely as a disjoint union of blocks, say $S=\bigsqcup_{j=1}^{m}B_{j}$. We define a bijection $\Phi_{S}$, analogous to the one in the proof of Proposition~\ref{prop:0102_0112}, which changes occurrences of ${\underline{0^{r}\,1\,0^{r}\,2\,0^{r}\ldots (s-1)\,0^{r}\,s}}$ in positions of $S$ into occurrences of ${\underline{0^{r}\,1\,1^{r}\,2\,2^{r}\ldots (s-1)\,(s-1)^{r}\,s}}$, as suggested by the schematic diagram in Figure~\ref{fig:Scheme_generalize_0102_0112}. Specifically, we set $\Phi_{S}\left(e\right)=e'$, where
\[
  e'_{i} = \begin{cases}
      e_{i-b-1}, & \text{if } i-a(r+1)-b\in S\text{ for some }1\leq a\leq s-1\text{ and } 0\leq b\leq r-1,\\
      e_{i}, & \text{otherwise}.
      \end{cases}
\]

\begin{figure}[htb]
	\noindent \begin{centering}
  \includegraphics[width=0.98\textwidth]{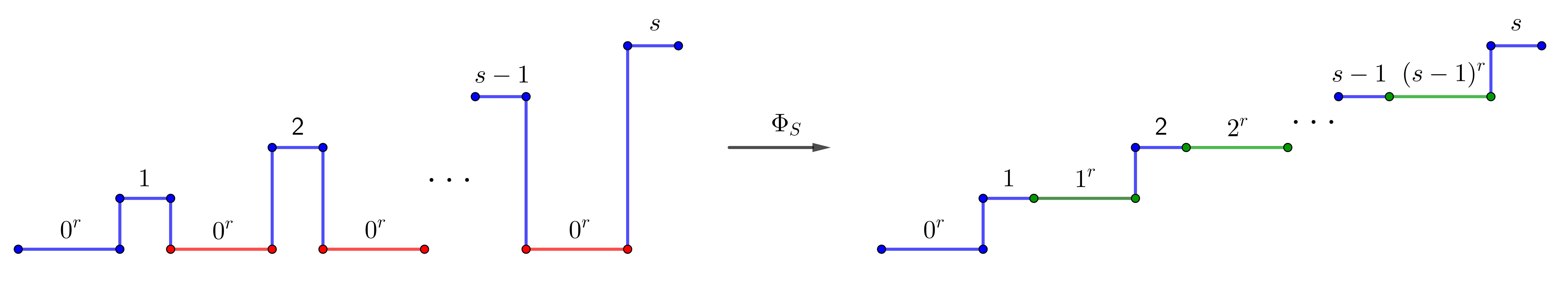}
		\par\end{centering}

	\protect\caption{Schematic diagram showing how $\Phi_{S}$ changes an occurrence of ${\underline{0^{r}\,1\,0^{r}\,2\,0^{r}\ldots (s-1)\,0^{r}\,s}}$ into an occurrence of ${\underline{0^{r}\,1\,1^{r}\,2\,2^{r}\ldots (s-1)\,(s-1)^{r}\,s}}$.\label{fig:Scheme_generalize_0102_0112}}
\end{figure}

\begin{thm}\label{thm:generalize_0102_0112} \sloppy For every $r\ge1$ and $s\geq 2$, the patterns ${\underline{0^{r}\,1\,0^{r}\,2\,0^{r}\ldots (s-1)\,0^{r}\,s}}$ and ${\underline{0^{r}\,1\,1^{r}\,2\,2^{r}\ldots (s-1)\,(s-1)^{r}\,s}}$ are super-strongly Wilf equivalent.
\end{thm}

Similarly, we can generalize Proposition~\ref{prop:2010_2110_2120} as follows.

\begin{thm}\label{thm:generalize_2010_2120_2110} \sloppy For every $r\ge1$ and $s\geq 2$, the patterns  ${\underline{s\,0^{r}\,(s-1)\,0^{r}\ldots 0^{r}\,10^{r}}}$, ${\underline{s\,(s-1)^{r}\,s\,(s-2)^{r}\,s\ldots s\,1^{r}s\,0^{r}}}$ and ${\underline{s\,(s-1)^{r}\,(s-1)\,(s-2)^{r}\,(s-2)\ldots 1^{r}1\,0^{r}}}$ are super-strongly Wilf equivalent.
\end{thm}

\subsection{A note regarding patterns extending $\underline{110}$ and $\underline{100}$}\label{subsec:extend_110_100}

\begin{defin} Let $r>s>0$. Suppose $p=\underline{p_{1}p_{2}\ldots p_{r}}$ and  $p'=\underline{p'_{1}p'_{2}\ldots p'_{s}}$ are consecutive patterns of
length $r$ and $s$, respectively. We say
\textit{$p$ extends $p'$ on the right} (respectively, \textit{left}) if
the leftmost (respectively, \textit{rightmost}) $s$ entries of any occurrence of $p$ in an inversion sequence form an occurrence of $p'$.
If $p$ extends $p'$ either on the left or on the right,
then we say \textit{$p$ extends $p'$}.
\end{defin}

\begin{exa} The patterns of length 4 extending $\underline{110}$ are $\underline{1102}$, $\underline{1101}$, $\underline{2201}$, $\underline{1100}$
and $\underline{2210}$ on the right; and $\underline{2110}$, $\underline{1110}$, $\underline{1220}$, $\underline{0110}$ and $\underline{0221}$ on the left.
The patterns of length 4 extending $\underline{100}$ are $\underline{1002}$, $\underline{1001}$, $\underline{2001}$, $\underline{1000}$
and $\underline{2110}$ on the right; and $\underline{2100}$, $\underline{1100}$, $\underline{1200}$, $\underline{0100}$ and $\underline{0211}$ on the left.
\end{exa}

We proved in Proposition~\ref{prop:110ss100} that $\underline{100} \stackrel{ss}{\sim} \underline{110}$, so it is natural to ask whether $p \stackrel{ss}{\sim} p'$ for some patterns $p$ and $p'$ of length 4 extending $\underline{100}$ and  $\underline{110}$, respectively. It follows from Theorem~\ref{Equiv4} that, in fact, there is a bijective correspondence between patterns of length 4 extending $\underline{100}$ and those extending $\underline{110}$,
preserving super-strong Wilf equivalence classes.

\begin{cor}\label{cor:correspondence_patterns_extending_100_110} Consider the bijection from patterns of length $4$ extending $\underline{100}$ to patterns of length $4$ extending $\underline{110}$ given by the following assignment:\vspace{-6pt}
\begin{multicols}{5}
\noindent $\underline{0100}\mapsto\underline{0110}$,\\
$\underline{1002}\mapsto\underline{1102}$,\\
$\underline{1200}\mapsto\underline{1220}$,\\
$\underline{0211}\mapsto\underline{0221}$,\\
$\underline{1000}\mapsto\underline{1110}$,\\
$\underline{1001}\mapsto\underline{1101}$,\\
$\underline{1100}\mapsto\underline{1100}$,\\
$\underline{2100}\mapsto\underline{2210}$,\\
$\underline{2001}\mapsto\underline{2201}$,\\
$\underline{2110}\mapsto\underline{2110}$.
\end{multicols}\vspace{-6pt}
\noindent Then, for each pair $p\mapsto p'$, we have that $p\stackrel{ss}{\sim}p'$.
\end{cor}

\nocite{*}

\bibliographystyle{plain}
\bibliography{consecI_bib}
\label{sec:biblio}

\end{document}